\documentclass[11pt]{amsart}
\usepackage{amsthm,amstext,amscd,amsbsy,amsxtra,latexsym}

\usepackage{cite}

 \usepackage{tabularx}
 \usepackage{booktabs} % For professional looking lines
 \usepackage{enumitem} % For controlled lists within tables

\usepackage[citation-order, msc-links]{amsrefs}

\usepackage{graphicx}
\usepackage{hyperref}
\hypersetup{
	colorlinks   = true,
	urlcolor     = blue,
	linkcolor    = blue,
	citecolor    = blue
}

\usepackage[labelsep=period]{caption}
 \usepackage[usenames,dvipsnames]{color}
\usepackage{booktabs}

\usepackage[capitalise]{cleveref}
\crefname{thm}{Thm.}{}
\crefname{prop}{Prop.}{}
\crefname{lem}{Lem.}{}
\crefname{cor}{Cor.}{}
\crefname{prob}{Problem}{}
\crefname{figure}{Fig.}{}

\usepackage{color}

 \newcommand{\Kb}{{{\bar K}}}

 \newcommand{\Z}{{\mathbb Z}}

 \newcommand{\Q}{{\mathbb Q}}
 \newcommand{\Pp}{{\mathbb P}}
 \newcommand{\CC}{{\mathbb C}}

 \newcommand{\Cc}{{\mathcal C}}
 
 \newcommand{\Ee}{{\mathcal E}}

 \newcommand{\KK}{{\mathcal K}}
 
 \newcommand{\OO}{{\mathcal O}}
 \newcommand{\Sc}{{\mathcal S}}

 \newcommand{\Xc}{{\mathcal X}}

 \newcommand{\gal}{{\text{Gal}}}
 
 \newcommand{\aut}{{\text{Aut}}}

 \newcommand{\rk}{{\text{rk}}}

 \newtheorem{thm}{Theorem}[section]
 \newtheorem{cor}[thm]{Corollary}
 \newtheorem{lema}[thm]{Lemma}

 \newtheorem{prop}[thm]{Proposition}

 \numberwithin{equation}{section}
 
 \usepackage{chngcntr}
 \numberwithin{table}{section}

\newcommand\iso{\cong}

\def\mod{\mbox{ mod }}
\def\deg{\mbox{deg }}

\usepackage{color}

\begin{document}

\title{The splitting field and generators of the  elliptic surface  $Y^2=X^3 +t^{360} +1$}

\author{Sajad Salami}
\address{Institute of Mathematics and Statistics, Rio de Janeiro State University, Rio de Janeiro, RJ, Brazil}
\email{sajad.salami@ime.uerj.br}

%	\author{Arman Shamsi Zargar}
%	\address{Department of Mathematics and Applications, University of Mohaghegh Ardabili, Ardabil, Iran}
%\email{zargar@uma.ac.ir}
\date{}

\newcommand{\I}{{\rm{i}}}

\begin{abstract}
The splitting field of an elliptic surface $\mathcal{E}/\mathbb{Q}(t)$ is the smallest finite extension $\mathcal{K} \subset \mathbb{C}$ such that all $\mathbb{C}(t)$-rational points are defined over $\mathcal{K}(t)$. In this paper, we provide a symbolic algorithmic approach to determine the splitting field and a set of $68$ linearly independent generators for the Mordell--Weil lattice of Shioda's elliptic surface $Y^2=X^3 +t^{360} +1$. This surface is noted for having the largest known rank 68 for an elliptic curve over $\mathbb{C}(t)$.
 Our methodology utilizes the known decomposition of the Mordell-Weil Lattice of this surface into Lattices of    ten rational elliptic surfaces and one $K3$ surface. We explicitly compute the defining polynomials of the splitting field, which reach degrees of 1728 and 5760, and verify the results via height pairing matrices and specialized symbolic software packages.
\end{abstract}
	


\maketitle
%\tableofcontents

%=================================================
\section{Introduction and the main result}

Let $k \subset  \CC$ be a number field.
Assume that $\Ee$ is an elliptic curve defined over $k(t)$ the rational function field such that 
$\Ee(\CC(t))$ the group of   $\CC(t)$-rational points of $\Ee$ is finitely generated.
The {\sf splitting field}
of $\Ee$ over $k(t)$  is the smallest subfield $\KK$ of $\CC$ containing $k$ 
%such that there is  a  set of generators in  $\Ee(K (t))$
for which  $\Ee (\CC (t))=\Ee (K(t))$. 
We notice that  $\KK| k$ is a Galois  extension with  finite Galois group  $G=\gal(\KK | k)$
and  the set of $G$-invariants elements of $\Ee (\KK(t))$ are the $\Ee (k(t))$-rational points.

\iffalse
In relation with  Mordell-Weil Lattices, the  splitting field can also be defined as follows.
There is a natural action of $\gal(\CC | k)$ on $\Ee(\CC (t))$ that preserves the height pairing  
$\left\langle ,  \right\rangle  $ and gives a Galois representation 
$$\rho :  \gal(\CC | k) \longrightarrow  \aut (\Ee (\CC (t)), \left\langle ,  \right\rangle ).$$
Since  the height pairing   $\left\langle ,  \right\rangle  $ on $\Ee (\CC(t))$
is positive definite up to torsion, so ${\rm Im}(\rho)$ the image of $\rho$ is a finite group. 
Hence, in the terminology of  Galois theory, the splitting field $K$ is exactly  the extension 
of $k$ which corresponds to  $\ker (\rho)$, and  we have  $\gal(\KK | k)= {\rm Im}(\rho).$
\fi 

In \cite{Shioda1999a}, Shioda stated that for any elliptic curve of the form 
$Y^2 =X^3 + t^m +a$, with $m\geq 1$ and $a \in \Q^*$,  the splitting field $K$
is a cyclic extension of a cyclotomic field and probably the minimal vectors of its 
Mordell-Weil lattice give rise  to a subgroup of finite index in the unit group of $K$. 
As a concrete example, he treated with the case $m=6$ and $a=-1$.
Then, in \cite{Shioda1999}, Shioda considered certain elliptic curves over the rational function field  $\Q(t)$ with Mordell-Weil lattices of type $E_6, E_7$ and $E_8$.
%The interested readers can find more details on the elliptic surfaces and thei Mordell-Weil lattices in
%\cite{Shioda1990a, Schuett2019, Silverman1994}.

In this paper, we  consider the  Shioda's famous  elliptic curve
\begin{equation}
	\label{shi-eq1}
	\Ee : Y^2=X^3 + t^{360} +1,
\end{equation}
defined over $\Q(t)$ which was stated in \cite{Shioda1992a} that its 
rank is equal to $68$ over $\CC(t)$, but without any proof. 
The structure of   group $\Ee (\CC(t))$ is extensively studied  
by H. Usui, in  \cites{Usui2008}, using the theory of Mordell-Weil Lattices.
%In particular,  it is proved that the assertion of Shioda on the rank of $\Ee$ is correct.
Based on  a computer implementation of
Hodge cycles, in \cite[Sec. 15.13]{movasati2021course}, H. Movasati  demonstrated that  the rank of $\Ee(\CC(t))$ is $68$. We note that this is the largest known rank of elliptic surfaces over $\CC(t)$ up to now.
In a different, but interesting way,  it is also  proved in \cite{Chahal2000} that Shioda's curve has exact rank $68$ over $\CC(t)$  using some Galois Theory and   Linear Algebra techniques.
The key ingredient of the proof is  the facts that  $\Ee$ has   constant moduli 
but it is non-constant over $\CC(t)$, since it is a $j$-invariant  zero  elliptic curve.

From the perspective of ``symbolic computation", this surface presents a unique opportunity: providing a complete, exact set of $68$ linearly independent sections and identifying the minimal polynomials of the fields containing their coefficients.
Thus, the main aim of this work is to  determine the splitting field $\KK$ of $\Ee$ and  provide a list of $68$ linearly independent points in $\Ee(\KK(t))$. 	Our computational pipeline leverages the fact that $\mathcal{E}$ is an isotrivial elliptic surface with $j$-invariant zero. This allows for a decomposition of the $68$-dimensional Mordell-Weil group into smaller subspaces corresponding to the ranks of simpler   surfaces
 $\mathcal{E}_{a,b} : y^2 = x^3 + v^a(v^b+1)$.  

	To achieve these results, we utilized a combination of specialized tools:
\begin{itemize}
	\item \textsf{Maple  \cite{maple}}: For the manipulation of polynomial ideals and the calculation of fundamental polynomials using the {\sf PolynomialTools} and {\sf PolynomialIdeals}   packages.
	\item \textsf{Pari/GP \cite{PARI2}}: For number field arithmetic, including the computation of compositum fields and minimal polynomials of high degree.
\end{itemize}

 . Here is the main result of this paper.

\iffalse
 We also keep some notations throughout this paper, as collected in the table \ref{tab:notation}.
\begin{table}[htbp]
	\caption{Notation}\label{tab:notation}
	\begin{tabular}{|l|l|l|}
		\toprule
		%		$n$ &  $1$ & $2$ \\ 
		%		\midrule
		$\I= \sqrt{-1}$ & $\epsilon_1=2+ \sqrt{3}$ &    $\epsilon'_1=2- \sqrt{3}$  \\ [5pt]
		$\displaystyle\zeta_3 = \frac{\mathrm{i} \sqrt{3} -1}{2}$ & $\epsilon_2  = 11 \sqrt{2} + 9 \sqrt{3}$  &    $\epsilon'_2  = -11 \sqrt{2} + 9 \sqrt{3}$   \\ [5pt]
		$\displaystyle\zeta_5 =\frac{\sqrt{5}-1 + \I \sqrt{2} \sqrt{5 +\sqrt{5}}}{4}$ &  $\epsilon_3  =  \sqrt{2} + 5 \I$ & $\epsilon'_3  =  \sqrt{2} - 5 \I$  \\ [5pt]
		$\displaystyle\zeta_6 =  \frac{1+\mathrm{i} \sqrt {3} }{2}$ &  $\epsilon_4 =1-\zeta_{12}$	& $\beta_0= 2^{\frac{1}{6}}$ \\ [5pt]
		$\displaystyle\zeta_8 =  \frac{\sqrt{2} ( 1+\mathrm{i} )}{2}$ &  $\displaystyle \epsilon_5 =\frac{\sqrt {3}-\sqrt {5}}{2}(\zeta_{12}+ \zeta_{12}^{10})$ &
		$\beta_1= (2+ \sqrt{3})^{\frac{1}{4}}$   \\ [5pt] 
		$\displaystyle\zeta'_8 =\frac{\sqrt{2} ( 1-\mathrm{i} )}{2}$ & $\displaystyle \epsilon_6  =\frac{1 + \sqrt {5}}{2} \zeta_{12}$  & $\beta_2=(3-2\sqrt{3})^{\frac{1}{4}} $ \\ [5pt]
$\zeta_{9}=\frac{( i \sqrt{3}-1 )(-4( i \sqrt{3})^{1/3}+1)}{4} $	& 		$\displaystyle\zeta_{12} = \frac{\I + \sqrt {3}}{2}$  & \\ [5pt]
		\bottomrule
	\end{tabular}
\end{table}
\fi

\begin{thm}
	\label{main0}
	The splitting field $\KK$ of the Shioda's elliptic curve \eqref{shi-eq1} is the compositum field of two number fields defined by two polynomials of degree 1728 and 5760.
%	equal to 
%	$$\KK:= \Q (\zeta_3, \zeta_5, \zeta_9, \zeta_{12} , \zeta_{24}, 2^{1/9}, 2^{1/12}, 5^{1/24}, \cdots ),$$
%	and 
	The Mordell-Weil group $\Ee(\KK(t))$ is generated by $P_1,\cdots, P_{68}$, as listed below:
	\begin{itemize}

				\item[(1)] Two points  $P_1,$ and $ P_2$  resulting from $Q_1, Q_2\in \Ee_{2, 1}(\KK_1(v))$,  given by Theorem \ref{main7},  are  
			$$P_1=(-t^{120}, 1), \text{and}\
				P_2 = (1,  - t^{180} ).$$
			%	 \left( -\zeta_6 \left(\frac{ t^{360} + \frac{4}{3}}{t^{240}}\right) , \,
			%	 \I \sqrt{3}\left(\frac{t^{360} + \frac{8}{9}}{t^{360}}\right) \right),$$
	%	 obtained from point $Q_1$ and $Q_2$
	%	  given by Theorem \ref{main7}.
			
				\item[(2)] Two points  $P_{3},$ 
				and $ P_{4}$ resulting from  $Q'_1, Q'_2 \in \Ee_{3, 1}(\KK_1(v))$,  given by Theorem \ref{main7},  are 
				$$
			P_3=(-\zeta_3 t^{120}, 1),  \text{and}\
			P_4 =  (-\zeta_3, t^{180}).$$
			% \left( -\zeta_6   \left(\frac{4}{3}t^{360} +1 \right) , \, \I \sqrt{3} t^{180} \left(  \frac{8}{9}t^{360} +1 \right)  \right).
		%	$$  		obtained from point $Q'_1$ and $Q'_2$  given by Theorem \ref{main7}.
			
			\item[(3)] Four points    $Q_1, \cdots, Q_4 \in \Ee_{1, 2}(\KK_2(v))$ of the form
			 $Q_j=(  a_j v +b_j,      c_j v + d_j  )$,
		where the coefficients are	 provided by Theorem  \ref{main6}, lead  to the points
			$P_5, \cdots, P_8$ of the form
			$$P_{j+4}= (  a_j t^{120} +b_j t^{-60}, \quad    c_j t^{90} + d_j t^{-90}), \  \text{for} \ j=1,\cdots, 4. $$

		\iffalse	
			 \begin{align*}
				P_5 &= \left( \beta_1^{-2}t^{-60} - t^{120}, \quad \beta_1^{-3}t^{-90} + \frac{1}{2}(\beta_1^3 - 3\beta_1^{-1})t^{90} \right) \\
				P_{6} &= \left( \beta_2^{-2}t^{-60} - t^{120}, \quad \beta_2^{-3}t^{-90} + \frac{1}{2}(\beta_2^3 - 3\beta_2^{-1})t^{90} \right) \\
								P_{7} &=  \left( -(\beta_1^{-2}t^{-60} + t^{120}), \quad i\beta_1^{-3}t^{-90} - \frac{i}{2}(\beta_1^3 - 3\beta_1^{-1})t^{90} \right) \\
				P_{8} &= \left( -\zeta_3(\beta_1^{-2}t^{-60} + t^{120}), \quad -i\beta_1^{-3}t^{-90} + \frac{i}{2}(\beta_1^3 - 3\beta_1^{-1})t^{90} \right)
			\end{align*}
		\fi	
		
				\item[(4)] Four points     $Q'_1, \cdots, Q'_4 \in \Ee_{3, 2}(\KK_2(v))$  of the form
				$Q_j=(  a_j v +b_j v^2 ,     c_j v^2 + d_j v^3 )$,
				where the coefficients are	 provided by Theorem  \ref{main6}, lead  to the points
				$P_9, \cdots, P_{12}$ of the form
			$$P_{j+8}= (  a_j  +b_j t^{180}, \quad    c_j t^{90} + d_j t^{270}), \  \text{for} \ j=1,\cdots, 4. $$

			\iffalse	 
				 
		\begin{align*}
		P_9 &= \left( \beta_1^{-2} t^{180} - 1, \quad \frac{1}{2}(\beta_1^3 - 3\beta_1^{-1})t^{90} + \beta_1^{-3}t^{270} \right) \\
		P_{10} &=  \left( -(\beta_1^{-2} t^{180} + 1), \quad -\frac{i}{2}(\beta_1^3 - 3\beta_1^{-1})t^{90} + i\beta_1^{-3}t^{270} \right) \\
		P_{11} &= \left( \beta_2^{-2} t^{180} - 1, \quad \frac{1}{2}(\beta_2^3 - 3\beta_2^{-1})t^{90} + \beta_2^{-3}t^{270} \right) \\
		P_{12} &= \left( -\zeta_3 (  \beta_1^{-2}t^{180}+1), \quad \frac{i}{2}(\beta_1^3 - 3\beta_1^{-1})t^{90} - i\beta_1^{-3}t^{270} \right)
	\end{align*}
	
	\fi
			
				\item[(5)] Four points   $Q_j=(a_j v+ b_j, v^2 + c_j v + d_j )\in \Ee_{2, 2}(\KK'_2(v))$,   give the followings:
				$$P_{j+12} = \left( a_j t^{60} + b_j t^{-120}, \quad t^{180} + c_j + d_j t^{-180} \right), \  \text{for} \ j=1,\cdots, 4,$$
				where the coefficients $a_j, b_j, c_j,$ and $d_j$'s are given in Theorem  \ref{main5}.
			%	 \begin{align*}
		%		 	P_{13} &= \left( \beta_0^{-2} t^{-120}, \; t^{180} + \frac{1}{2} t^{-180} \right), \\
		%		 	P_{14} &= \left( \zeta_3^2 \beta_0^{-2} t^{-120}, \; t^{180} + \frac{1}{2} t^{-180} \right), \\
			%	 	P_{15}  &= \left( i\sqrt{3}\beta_0 t^{60} + 2\beta_0 t^{-120}, \; t^{180} - 3i\sqrt{3} - 4 t^{-180} \right), \\
		%		 	P_{16} &= \left( i\sqrt{3}\zeta_6\beta_0 t^{60} + 2\zeta_3^2\beta_0 t^{-120}, \; t^{180} + 3i\sqrt{3} - 4 t^{-180} \right).
			%	 \end{align*}

				\item[(6)] Six points  $P_{17}, \cdots, P_{22}$  resulting from $Q_1, \cdots, Q_6 \in \Ee_{1, 3}(\KK_3(v))$ are of the   form
				$$P_{j+16}= \left(  a_j t^{80} + b_j  t^{-40},
		 t^{180} + d_j t^{60} + e_j t^{-60} \right)  \ \text{for } j=1,\cdots, 6,$$
			where the constants $a_j, b_j,  d_j$ and $ e_j$'s are given by Theorem \ref{main4}.
	%		  \cite[Points-68]{Shioda360-Codes}.

				\item[(7)] Six points  $P_{23}, \cdots, P_{28}$  resulting from $Q'_1, \cdots, Q'_6 \in \Ee_{2, 3}(\KK_3(v))$ are of the   form
			$$P_{j+22}= \left( b_j t^{160} + a_j t^{40},
			e_j t^{240} + d_j t^{120} + 1 \right)  \ \text{for } j=1,\cdots, 6,$$
			where the constants $a_j, b_j,  d_j, e_j$ are given 
		%	in   \cite[Points-68]{Shioda360-Codes}.
	 	by Theorem \ref{main4}.

			\item[(8)] Eight points $P_{29}, \cdots, P_{36}$   resulting from $
			Q_1, \cdots, Q_8 \in \Ee_{1, 4}(\KK_4(v))$  are of the   form:
			$$P_{j+28}=\left(  \frac{t^{180}+ a_j t^{90} + b_j}{u_j^2 t^{30}}, \quad
			\frac{t^{270}+ c_j t^{180} + d_j t^{90} + e_j}{u_j^3 t^{45}}\right), \ \text{for } j=1,\cdots, 8, $$
			where the constants $a_j, b_j, c_j,  d_j, e_j$ and $u_j$'s are given in Theorem \ref{main3}.
		
					\item[(9)]
			Eight points $P_{37}, \cdots, P_{44}$ resulting from $\Ee_{0,5}(\KK_5(v))$ are of the  following form
			$$P_{j+36}=\left(\frac{t^{144}+ a_j t^{72} + b_j}{u_j^2 }, \quad
			\frac{t^{216}+ c_j t^{144} + d_j t^{72} + e_j}{u_j^3}\right), \ \text{for } j=1,\cdots, 8, $$
			where the constants $a_j, b_j, c_j,  d_j, $ and $e_j$ are given by part (2) of Theorem \ref{main2}.

			\item[(10)]
			Eight points $P_{45}, \cdots, P_{52}$   resulting from $\Ee_{1,5}(\KK_5(v))$  are of the following form 
			$$P_{j+44}=\left(\frac{b_j t^{144}+ a_j t^{72} + 1}{u_j^2 t^{24}}, \quad
			\frac{e_j t^{216}+ d_j t^{144} + c_j t^{72} + 1}{u_j^3 t^{36}}\right), \ \text{for } j=1,\cdots, 8, $$
			where the constants $a_j, b_j, c_j,  d_j$ and $  e_j$ are given by part (3) of Theorem \ref{main2}.
			
			\item[(11)] The 16 points $P_{53}, \cdots, P_{68}$ resulting from $\Ee' (\KK'(v)) $ have coordinates
		\begin{align*}
			x(P_{j}) & =   A_{4,j} t^{132} + A_{3,j} t^{96} + A_{2,j} t^{60} + A_{1,j} t^{24} + A_{0,j} t^{-12} \notag \\ 			 
			y(P_{j}) &=   B_{6,j} t^{198} + B_{5,j} t^{162} + B_{4,j} t^{126} +B_{3,j} t^{90} + B_{2,j} t^{58}
		+	B_{1,j} t^{18} + B_{0,j} t^{-18}, 
		\end{align*}
	for $j=53, \cdots, 68$, as in Theorem \ref{main1}.
	\end{itemize}
\end{thm}

In Table~\ref{tab:algorithm}, we provide and algorithmic description of the calculations to drive 
We have to mentioned that all of the 68 points are   listed in \cite[Points-68]{Shioda360-Codes}, and the defining polynomials of the splitting field $\KK$ are given in \cite[Pol-1728, Pol-11520]{Shioda360-Codes}.
% as well as   {\sf{SageMath}}~\cite{sagemath}.

The structure  of   paper is organized as follows. In the next section, we describe an   sketch and  of the proof  of Theorem \ref{main0}.
In Section \ref{Eab},  we provide some results on the family of elliptic surfaces $\Ee_{a,b}$ over the rational function field $\Q(v)$.
The splitting field and generators of the elliptic surfaces $\Ee_{2,1}$ and $\Ee_{3,1}$ have been determined in Section
\ref{A2-(2,1)}.  The splitting field and generators of  $\Ee_{1,2}$ and $\Ee_{3,2}$   have been computed in Section
\ref{D4-(1,2)}.  The splitting field and generators of $\Ee_{2,2}$ have been calculated in Section
\ref{A2+A2-(2,2)}. In Section \ref{E6-(1,3)-(2,3)}, we determined  the splitting field and generators of $\Ee_{1, 3}$  and $\Ee_{2, 3}$.
 The splitting field and generators of $\Ee_{1, 4}$  have been computed in Section
\ref{E8-(1,4)}.  The  splitting field and generators of  $\Ee_{0,5}$ and $\Ee_{1,5}$ have been treated in Section
\ref{E-{0,5}}. In Section \ref{K3}, we treat the the  splitting field and generators of the  $K3$ surface $\Ee'$.
 Finally, in Section \ref{proof}, we complete the proof of  main Theorem \ref{main0}.

%+++++++++++++++++++++++++++++++++++++++++++++++++++++++++++++++++++++++++++++++++++==
\section{Sketch of the proof of  Theorem \ref{main0}}

Let $\Ee$ be an elliptic curve over $\CC(t)$ defined by 
$y^2=x^3+ f(t^6)$, where $f$ is a non-constant polynomial  with rational coefficients and $f(0)\neq 0$. Defining
$V:=\Ee(\CC(t)) \otimes \Q$  and  considering the map
$\phi: V \rightarrow V$ given by 
$(x, y) \mapsto \zeta_3 \cdot P := (\zeta_3 x, - y)$
which	is a linear map satisfying $\phi^6=id$,  one see that $V$ has a structure of  a $\Q(\sqrt{-3})$-vector space defined by
$(a-b  \zeta_3)\cdot P= a \cdot P + b \cdot \phi(P)$ for all $a, b \in \Q$ and $P=(x,y) \in \Ee(\CC(t)).$ This implies that if $P$ is of infinite order then
$P$ and $\phi(P) $ are linearly independent over $\Q$. 
Define the map $\psi: \CC(t) \rightarrow \CC(t)$ by $t\mapsto - \zeta_3  t$ which satisfies $\psi^6=id$ and for $\CC(t^6)=\{ x\in \CC(t): \psi(x)=x\}$  we have 
$[\CC(t): \CC(t^6)]=6$. Extending the map $\psi$ by 
$((x, y)\otimes a) \mapsto ((\psi(x), \psi(y))\otimes a)$,
we obtain 
$\bar{\psi}: V \rightarrow V$   which acts on $V$ and 
has order $6$  as a $\Q(\sqrt{-3})$- linear automorphism.
Its six eigenvalues are contained in the subset
$A=\{(-\zeta_3)^i:   0 \leq i \leq 5\}=\{\pm 1, \pm \zeta_3, \pm \zeta_3^2 \} \subset \Q(\sqrt{-3}).$
Denoting  the eigen space corresponding to the eigenvalue $(-\zeta_3)^i$ by $V_i$, we have
$V_i \cong \Ee_i(\CC(t^6))\otimes \Q $ where $ \Ee_i: y^2=x^3+ t^{6 i} f(t^6)$.
Thus, we have 
$$\rk (\Ee(\CC(t))) =\dim_\Q(V)=\sum_{i=0}^{5} \dim_\Q(V_i)=
\sum_{i=0}^5 \rk (\Ee_i(\CC(t^6))).$$

The method described 	above  is introduced in \cite{meijer}  and used to obtain high rank elliptic surfaces over $\Q(t) $ and $\bar{\Q}(t)$. 
Then, it is  used in \cites{Chahal2000, meijer} to study on Mordell-Weil group of $j=0$ elliptic curves over rational function fields. In particular, \cite{Chahal2000}, it is showed that the rank of Shioda's elliptic curve is exactly equal to $68$. 
In order to describe  the  idea of their proof, let us to consider the $K3$-surface $\Ee': y^2= x^3 + v\,(v^{10} + 1)$   and  
the rational elliptic surfaces $\Ee_{a, b}: y^2= x^3 + v^a (v^b+1)$ over $\CC(v)$ as follows:
\begin{align}
		\Ee_{0,\, 5}:\ &  y^2= x^3 + v^5 + 1,  &  
		\Ee_{1,\, 5}:\ &  y^2= x^3 + v\,(v^5 + 1), \notag \\
		\Ee_{1,\, 4}:\ &  y^2= x^3 + v\,(v^4 + 1),  &  
		\Ee_{1,\, 3}:\ &  y^2= x^3 + v\,(v^3+1),   \notag \\
		\Ee_{2,\, 3}:\ &  y^2= x^3 + v^2\,(v^3+1), &
		\Ee_{2,\, 2}:\ &  y^2= x^3 + v^2\,(v^2+1),  \notag \\
		\Ee_{3,\, 2}:\ &  y^2= x^3 + v^3\,(v^2+1),  &
		\Ee_{1,\, 2}:\ &  y^2= x^3 + v \,(v^2+1),\notag \\
		\Ee_{3,\, 1}:\ &  y^2= x^3 + v^3\,(v+1),  &
		\Ee_{2,\, 1}:\ &  y^2= x^3 + v^2\,(v+1). 	\label{rak1}
\end{align}
Here, the parameter $v$ is a certain power of $t$ for each case such that the
degree of field extension $\CC(t) | \CC(v)$ is a multiple  of $6$. 
Then,   the  proof of Proposition 4.2 in \cite{Chahal2000} shows that
\begin{align}
		\rk (\Ee ) & =   \rk (\Ee')+\rk (\Ee_{0,\, 5}) + \rk (\Ee_{1,\, 5}) +\rk (\Ee_{1,\, 4})+\rk (\Ee_{2,\, 3})+\rk (\Ee_{1,\, 3}) \notag \\ 
		& \hspace{0.3cm }+ \rk (\Ee_{2,\, 2}) + \rk (\Ee_{3,\, 2})+ \rk (\Ee_{1,\, 2}) +\rk (\Ee_{3,\, 1})+ \rk (\Ee_{2,\, 1}) \notag \\ &=16+8 + 8 + 8 + 6 + 6 + 4 + 4 + 4 + 2 + 2=68, 	\label{rak2}
\end{align}
where $\rk (X)$ is the rank of vector spaces $X \otimes \Q$  for 
$X= \Ee (\CC(t)),$ or $\Ee'(\CC(v))$, and  or $\Ee_{a, b}(\CC(v))$.
We note that for  integers $a$ and $b$ satisfying
$0 \leq a, b, \leq 5, \ \gcd(a,b)=1, \ \text{and}\  a+b \leq 6, $
 the elliptic 
surface $ \Ee_{a, b}$   is birational to $\Ee_{6-(a+b), b}$   under the map
$$\phi:   (x(v), y(v), v) \mapsto \left( v^2 x(1/v) , v^3 y(1/v), 1/v \right)$$  with the  inverse map 
$$\phi^{-1}:   (x(v), y(v), v) \mapsto \left( x(1/v)/v^2  , y(1/v)/v^3, 1/v \right).$$
% with the inverse 
% $\phi^{-1}: 
% (x(v), y(v), v) \mapsto \left( v^{-2} x(1/v), v^{-3 } y(1/v), 1/v\right).$
This implies the following isomorphisms of Mordell-Weil lattices:
\begin{align*}
	\Ee_{0, 5}(\CC(v))  \cong  \Ee_{1, 5}(\CC(v)),\quad & \Ee_{1, 3}(\CC(v)) \cong \Ee_{2,3}(\CC(v)),\\
	\Ee_{3, 2}(\CC(v))  \cong \Ee_{1, 2}(\CC(v)),\quad &  \Ee_{3, 1}(\CC(v))  \cong \Ee_{2, 1}(\CC(v)). 
\end{align*}

In below,   we describe an algorithmic approach for the   Theorem \ref{main0}.  
%By  Shioda's results \ref{shi-thm}, to determine the splitting field $\KK_n$ of $\Ee_n$ and a set of the linearly independent generating points of $\Ee_n ( \KK_n )$, we  will do the  steps provided in Table~\ref{Tab3}

\begin{table}[htbp]
	\centering
	\caption{Algorithm for computation on $\mathcal{E}(\mathcal{K})$} 
	\label{tab:algorithm}
	\begin{tabularx}{\textwidth}{@{} l X @{}} 
		\toprule
		& \textbf{Computing the Splitting Field and Generators of $\mathcal{E}(\mathcal{K})$} \\ 
		\midrule
		\textbf{Step 1:} & Determining the splitting field and linearly independent generators of\\
		& $\mathcal{E}_{a,b}: y^2=x^3+ v^a (v^b+1)$ over $\mathbb{Q}(v)$ for each pair: \\
	& 	$(a,b) \in \{ (2, 1), (1,2), (2, 2), (1, 3), (2, 3) , (1, 4) , (1, 5), (0, 5) \}$.
		
		\begin{itemize}[leftmargin=*, nosep, topsep=2pt]
			\item Take points (sections) of the elliptic surface of the form:
			
			$(x(v), y(v)) = (a_0 + a_1 v + a_2 v^2, b_0 + b_1 v + b_2 v^2 + b_3 v^3)$.
			\item Substitute into the equation of $\mathcal{E}_{a,b}$ to generate a system of equations in $a_i$ and $b_j$ defining an ideal in $\mathbb{Q}[a_0, a_1, a_2, b_0, b_1, b_2, b_3]$.
			\item Find the fundamental polynomial of the above ideals using the command \textsf{UnivariatePolynomial} from the \textsf{PolynomialIdeals} package in \textsf{Maple} and factor it into linear factors.
			\item Use \textsf{Pari/GP} to find a defining minimal polynomial  of the splitting field $\mathcal{K}_{b}$.
			\item Select appropriate roots to obtain linearly independent generators of $\mathcal{E}_{a,b}(\mathcal{K}_{b})$.
		\end{itemize} \\
		\addlinespace
		\textbf{Step 2:} & Determining corresponding points $(x'(v), y'(v))$ on the partner surface 
		$\mathcal{E}_{6-(a+b), b}(\mathcal{K})$ for pairs $(a,b) \in \{ (2, 1), (1,2), (1, 3), (0, 5) \}.$ \\
		\addlinespace
		\textbf{Step 3:} & Determining the splitting field $\mathcal{K}_{10}$ and linear generators of the elliptic K3 surface $\mathcal{E}_{1,10}(\mathcal{K}_{10})$.\\
		\textbf{Step 4:} & Determining the splitting field $\mathcal{K}$ and linear generators of $\mathcal{E}(\mathcal{K})$:
		\begin{itemize}[leftmargin=*, nosep, topsep=2pt]
			\item Use \textsf{Pari/GP} to find a defining minimal polynomial of the compositum of all fields $\mathcal{K}_{(b)}$.
			\item Transform the points $(x(v), y(v)) \in \mathcal{E}_{(a,b)}(\mathcal{K}_{b})$ into points belonging to $\mathcal{E}(\mathcal{K})$ using the transformations defined in \ref{lem:poly_transform}.
				\item Transform the points $(x(v), y(v)) \in \mathcal{E}_{(1,10)}(\mathcal{K}_{10})$ into points belonging to $\mathcal{E}(\mathcal{K})$ using the transformations defined in \ref{lem:poly_transform}.
		\end{itemize} \\
		\bottomrule
	\end{tabularx}
\end{table}

\section{Some results on the elliptic surfaces $\Ee_{a,b}$}
\label{Eab}
In this section, we provide some results on the family of elliptic surfaces defined over the rational function field $\Q(v)$ by the Weierstrass equation:
\begin{equation} \label{eq:main}
	\Ee_{a,b}: y^2 = x^3 + v^a(v^b+1)
\end{equation}
where $0 \leq a, b\leq 11$ are  coprime integers such that  $1   \leq a+b \leq 12$. This family is isotrivial with $j$-invariant identically zero.
%The geometric classification of these surfaces is determined by the degree of the discriminant 
% $ \Delta(v) = -432 v^{2a} (v^b+1)^2$. 
If $a+b \leq 6$, then  surface  $\Ee_{a,b}$ is a rational elliptic surface, and if
 $6 < a+b \leq 12$, then it is an elliptic K3 surface.

\subsection{Birational map between $\Ee_{a,b}$}

We first establish that every surface in this family is birationally equivalent to a "partner" surface obtained by inverting the base parameter.

\begin{thm} 
	\label{thm:birational}
	Let $n = \lceil (a+b)/6 \rceil$. The elliptic surface $\Ee_{a,b}$ is birationally equivalent to the surface $\Ee_{a', b}$ where $a' = 6n - (a+b)$.
\end{thm}

\begin{proof}
	We define a birational map $\phi: \Ee_{a,b} \dashrightarrow \Ee_{a', b}$ via a change of coordinates on the base $\Pp^1$ and a weighted change of coordinates on the fiber.
	Let $v = u^{-1}$. Substituting this into \eqref{eq:main}:
	\[
	y^2 = x^3 + u^{-a}(u^{-b}+1) = x^3 + u^{-(a+b)}(1+u^b).
	\]
	To return to a minimal integral Weierstrass equation, we scale the coordinates to clear the denominator. We multiply the equation by $u^{6n}$:
	\[
	u^{6n} y^2 = u^{6n} x^3 + u^{6n - (a+b)}(1+u^b).
	\]
	We define the new coordinates $(X, Y)$ such that $Y^2 = (u^{3n} y)^2$ and $X^3 = (u^{2n} x)^3$. This implies the transformation:
	\[
	X = u^{2n} x, \quad Y = u^{3n} y, \quad u = v^{-1}.
	\]
	The transformed equation is:
	\[
	Y^2 = X^3 + u^{6n - (a+b)}(1+u^b),
	\]
	which is precisely the defining equation for $\Ee_{a', b}$ with $a' = 6n - (a+b)$. Thus, the map 
	 $$\phi(x, y, v) = (v^{2n} , v^{3n} y, v^{-1})$$ is a 
	 rational and admits the rational inverse
	  $$\phi^{-1}(X, Y, u) = (u^{-2n}X, u^{-3n}Y, u^{-1}).$$
	  the surfaces are birationally equivalent.
\end{proof}

%\begin{rem}
%  For rational surfaces ($n=1$), the partner parameter is $a' = 6-(a+b)$, and 
%  for K3 surfaces ($n=2$), the partner parameter is $a' = 12-(a+b)$.
%\end{rem}

\subsection{Singular Fibers and Mordell-Weil Lattice of Rational $\Ee_{a,b}$}

In this section, we determine the structure of the Mordell-Weil group $E(\CC(v))$ as a lattice. Since the surfaces $\Ee_{a,b}$ are rational elliptic surfaces with a section, the Néron-Severi lattice is unimodular \cite[Prop. 7.1]{Schuett2019}. Consequently, the structure of the Mordell-Weil lattice is fully determined by the configuration of singular fibers via the trivial lattice $T \subset NS(\Ee_{a,b})$ \cite[Thm. 8.6]{Schuett2019}.

The singular fibers are located at the zeroes of the discriminant, which are determined by the vanishing of the coefficient $B(v) = v^a(v^b+1)$. Let $n = \lceil (a+b)/6 \rceil$. The minimal model has singular fibers potentially at $v=0$, the roots of $v^b=-1$, and $v=\infty$.

\begin{thm}
	\label{thm:lattice_structure}
	Let $\Ee_{a,b}$ be a rational elliptic surface in the family defined by \eqref{eq:main}. The Mordell-Weil rank is given by the Shioda-Tate formula $\rk(\Ee_{a,b}) = 8 - \sum (m_v - 1)$, and the group $\Ee_{a,b}(\CC(v))$ typically admits the structure of the dual lattice $L^\vee$, where $L$ is the narrow Mordell-Weil lattice. The specific ranks and isomorphism types for the surfaces under consideration are:
	\begin{enumerate}
		\item If $(a,b) \in \{(2,1), (3,1)\}$, the rank is $2$ and $\Ee_{a,b}(\CC(v)) \cong A_2^\vee(1/2)$.
		\item If $(a,b) \in \{(1,2), (3,2)\}$, the rank is $4$ and $\Ee_{a,b}(\CC(v)) \cong D_4^\vee$.
		\item If $(a,b) = (2,2)$, the rank is $4$ and $\Ee_{a,b}(\CC(v)) \cong A_2^\vee \oplus A_2^\vee$.
		\item If $(a,b) \in \{(1,3), (2,3)\}$, the rank is $6$ and $\Ee_{a,b}(\CC(v)) \cong E_6^\vee$.
		\item If $(a,b) \in \{(1,4), (1,5), (0,5), (0,6)\}$, the rank is $8$ and $\Ee_{a,b}(\CC(v)) \cong E_8$.
	\end{enumerate}
\end{thm}

\begin{proof}
	We determine the root lattice $T = \bigoplus T_v$ generated by fiber components disjoint from the zero section. The fiber type at a valuation $v$ with order $k$ contributes a root lattice determined by Tate's Algorithm:
	\begin{itemize}
	\item		If $k=1,$ then the fiber is of Type $II$ ($\emptyset$); 	 
	\item		If	 $k=2,$ then the fiber is of Type $IV$ ($A_2$);
	\item		If	 $k=3,$ then the fiber is of Type $I_0^*$ ($D_4$);
	\item		If	 $k=4, $ then the fiber is of Type $IV^*$ ($E_6$).
		\end{itemize}
		
	The roots of $v^b+1$ are simple zeroes ($k=1$), contributing Type $II$ fibers with no impact on the rank or lattice structure. We examine the fibers at $v=0$ (order $a$) and $v=\infty$ (order $a' = 6-(a+b)$ in the rational case $n=1$).
	
	\begin{itemize}

		\item \textbf{Case (2,1):} $a=2$ ($A_2$) and $a'=3$ ($D_4$). $T \cong A_2 \oplus D_4$. The rank is $8 - (2+4) = 2$. By \cite[Table 8.2, No. 32]{Schuett2019}, $E(\CC(v)) \cong A_2^\vee(1/2)$.
		
		\item \textbf{Case (1,2):} $a=1$ ($\emptyset$) and $a'=3$ ($D_4$). $T \cong D_4$. The rank is $8 - 4 = 4$. By \cite[Table 8.2, No. 9]{Schuett2019}, $E(\CC(v)) \cong D_4^\vee$.
		
		\item \textbf{Case (2,2):} $a=2$ ($A_2$) and $a'=2$ ($A_2$). $T \cong A_2 \oplus A_2$. The rank is $8 - (2+2) = 4$. By \cite[Table 8.2, No. 11]{Schuett2019}, $E(\CC(v)) \cong A_2^\vee \oplus A_2^\vee$.
		
		\item \textbf{Case (1,3):} $a=1$ ($\emptyset$) and $a'=2$ ($A_2$). $T \cong A_2$. The rank is $8 - 2 = 6$. By \cite[Table 8.2, No. 3]{Schuett2019}, $E(\CC(v)) \cong E_6^\vee$.
		
		\item \textbf{High Rank Cases:} For $(1,4)$, $(0,5)$, etc., we have $a, a' \leq 1$. There are no reducible fibers, so $T = \{0\}$. The rank is $8$ and $E(\CC(v)) \cong E_8$ \cite[Table 8.2, No. 1]{Schuett2019}.
	\end{itemize}
	The remaining cases follow by symmetry (swapping $a$ and $a'$).
\end{proof}

In  Table~\ref{Tab1}, we summarize the data for  rational elliptic surfaces $\Ee_{a,b}$ investigated in this paper. The term $\delta(k) = m_k - 1$ denotes the contribution of the fiber to the rank reduction.

\begin{table}[htbp]
	\centering
	\small
	\begin{tabular}{@{}cccccll@{}}
		\toprule
		$(a,b)$ & $a'$ & $\delta(a)$ & $\delta(a')$ & Rank & Fibre Types ($T$) & MWL Structure \\ \midrule
		$(2,1)$ & 3 & 2 & 4 & \textbf{2} & $A_2 \oplus D_4$ & $A_2^\vee(1/2)$ \\
		$(3,1)$ & 2 & 4 & 2 & \textbf{2} & $D_4 \oplus A_2$ & $A_2^\vee(1/2)$ \\
		$(1,2)$ & 3 & 0 & 4 & \textbf{4} & $D_4$ & $D_4^\vee$ \\
		$(2,2)$ & 2 & 2 & 2 & \textbf{4} & $A_2 \oplus A_2$ & $A_2^\vee \oplus A_2^\vee$ \\
		$(3,2)$ & 1 & 4 & 0 & \textbf{4} & $D_4$ & $D_4^\vee$ \\
		$(1,3)$ & 2 & 0 & 2 & \textbf{6} & $A_2$ & $E_6^\vee$ \\
		$(2,3)$ & 1 & 2 & 0 & \textbf{6} & $A_2$ & $E_6^\vee$ \\
		$(1,4)$ & 1 & 0 & 0 & \textbf{8} & $\emptyset$ & $E_8$ \\
		$(1,5)$ & 0 & 0 & 0 & \textbf{8} & $\emptyset$ & $E_8$ \\
		$(0,5)$ & 1 & 0 & 0 & \textbf{8} & $\emptyset$ & $E_8$ \\
		$(0,6)$ & 0 & 0 & 0 & \textbf{8} & $\emptyset$ & $E_8$ \\ \bottomrule
		\vspace{3pt}
	\end{tabular}
	\caption{Ranks and Lattice Structures for Rational Elliptic Surfaces $\Ee_{a,b}$. 
		Classification types are based on Theorem 8.8 in \cite{Schuett2019}.}
			\label{Tab1}
\end{table}

In the following result, we examine singular fibers and Mordell-Weil rank of the $K3$ surface $\Ee_{1,10}$

\begin{cor}
	\label{cor:E1_10}
	Consider the elliptic surface $\Ee_{1,10}$ defined by $y^2 = x^3 + v(v^{10}+1)$.
	\begin{enumerate}
		\item  The  surface is self-dual and birationally equivalent to itself under the map
		$$\phi: (x, y, v) \mapsto (v^{4}x, v^{6}y, v^{-1}).$$
		\item   The surface has exactly 12 singular fibers, all of Type $II$:
		\begin{itemize}
			\item At $v=0$: Valuation $a=1 \implies$ Type $II$.
			\item At the 10 roots of $v^{10}=-1$: Simple zeros $\implies$ Type $II$.
			\item At $v=\infty$: Valuation $12-(1+10)=1 \implies$ Type $II$.
		\end{itemize}
		\item  The Mordell-Weil rank is $r = 16$, and the Picard number is $\rho = 18$.
		% Consequently, while $\Ee_{1,10}$ is an elliptic K3 surface, it does not attain the maximal possible rank of 18 for complex K3 surfaces.
	\end{enumerate}
\end{cor}

\begin{proof}
	The surface is defined by a polynomial of degree $d=11$, hence it is a K3 surface.
	First, we compute the contribution from singular fibers. As established in item 2, all 12 singular fibers are of Kodaira Type $II$. Type $II$ fibers are irreducible curves (cusps), so the number of components is $m_v = 1$ for all $v \in R$. Thus, the correction term in the Shioda-Tate formula vanishes:
	\[
	\sum_{v \in R} (m_v - 1) = 12 \times (1-1) = 0.
	\]
 The surface falls into the class of Delsarte surfaces (of the form $y^2 = x^3 + t^d + 1$ or birational equivalents). According to the rank formula for such surfaces over $\mathbb{C}(v)$ (see \cite[Section 13.2.2]{Schuett2019}, one has
	\[
	r = 2d - 2 - 4 \left\lfloor \frac{d}{6} \right\rfloor= 16
	\]
	Finally, using the Shioda-Tate formula shows
 that $\Ee_{1, 10}$ has   Picard number $\rho = 18$
\end{proof}

\subsection{Transformation of points of $\Ee_{a,b}$ on $\Ee$}

We provide explicit formulas to map points between the surface, its partner, and the Shioda's elliptic  surface $\Ee: Y^2 = X^3 + t^{360} + 1$.

\begin{lema} 
	\label{lem:poly_transform}
	Let $n = \lceil (a+b)/6 \rceil$,  $a' = 6n-(a+b)$,  	 and $k = 360/b$.
	The following hold for any point $Q = (x(v), y(v))$     in $\Ee_{a,b}(\CC(v))$,  defined by
	\[ x(v) = \sum_{i=0}^{2n} A_i v^i, \quad y(v) = \sum_{j=0}^{3n} B_j v^j. \]

	\begin{enumerate}
		\item [(i)]  The point $Q'$ in $\Ee_{a', b}(\CC(u))$ corresponding to $Q$ (with $u=1/v$) is given by reversing coefficients:
		\[ x'(u) = \sum_{i=0}^{2n} A_{2n-i} \, u^i, \quad y'(u) = \sum_{j=0}^{3n} B_{3n-j} \, u^j. \]
		
		\item [(ii)]   The point $P=(X(t), Y(t))$ on $\Ee(\CC(t))$  derived from  $Q$ via $v=t^k$, is:
		\[ X(t) = \sum_{i=0}^{2n} A_i \, t^{k(i - \frac{a}{3})}, \quad Y(t) = \sum_{j=0}^{3n} B_j \, t^{k(j - \frac{a}{2})}. \]
		
		\item [(iii)]   The point $P'=(X'(t), Y'(t))$ on $\Ee(\CC(t))$  
		 derived from $Q'$ via $u=t^k$, is:
		\[ X'(t) = \sum_{i=0}^{2n} A_{2n-i} \, t^{k(i - \frac{a'}{3})}, \quad Y'(t) = \sum_{j=0}^{3n} B_{3n-j} \, t^{k(j - \frac{a'}{2})}. \]
	\end{enumerate}
\end{lema}

\begin{proof}
	Letting $u=1/v$, the rational map $\phi(x,y,v) = (u^{2n}x, u^{3n}y,u)$ applied to the polynomials results in the reversal of coefficients relative to the maximal degrees $2n$ and $3n$. 
	The Shioda map applies the weights $t^{-ka/3}$ and $t^{-ka/2}$ to $x(t^k)$ and $y(t^k)$. The term $A_i (t^k)^i$ becomes $A_i t^{ki - ka/3}$. The formula follows by linearity.
\end{proof}

% \newpage
 \section{The splitting field and generators of $\Ee_{2,1}$ and $\Ee_{3,1}$}
\label{A2-(2,1)}

In this section, we consider the elliptic surfaces defined by $\mathcal{E}_{3,1}: y^2 = x^3 + v^3(v+1)$ and $\mathcal{E}_{2,1}: y^2 = x^3 + v^2(v+1)$. 
The two surfaces are birational via the map $\phi: (x, y, v) \mapsto (v^2 x, v^3 y, 1/v)$. We focus on determining generators for $\mathcal{E}_{2,1}$ and extending the results to $\mathcal{E}_{3,1}$ via $\phi$. By Table \ref{Tab1} the Mordell-Weil lattices are isomorphic to the scaled dual lattice $A_2^\vee(1/2)$.

\begin{thm} \label{main7}
	The splitting field of the elliptic surfaces $\mathcal{E}_{2,1}$ and $\mathcal{E}_{3,1}$ is the cyclotomic field $\mathcal{K}_{1} = \mathbb{Q}(\I \sqrt{3})$, defined by $f_1(x)=x^2-x+1$.
	 
	The Mordell-Weil group $\mathcal{E}_{2,1}(\mathcal{K}_{1}(v))$ is generated by the sections $P_1$ and $P_2$, which constitute a basis for the Mordell-Weil lattice $A_2^\vee(1/2)$:
	\begin{equation}
		\label{Eq1}
		Q_1=(-v, v), \quad \text{and} \quad
		Q_2 = (-\zeta_3 v, v), \ \text{with } \zeta_3=\frac{\I \sqrt{3}-1}{2}.
	\end{equation}
	Via the birational map $\phi$, the generators for $\mathcal{E}_{3,1}(\mathcal{K}_{1}(v))$ are:
	\begin{equation} 	\label{Eq1-2}
		Q'_1=(-v, v^2), \quad \text{and} \quad
		Q'_2 = (-\zeta_3 v, v^2).
	\end{equation}	
\end{thm}

\begin{proof}
	We consider the surface $\mathcal{E}_{2,1}$ defined by $y^2 = x^3 + v^2 (v+1)$. The reducible singular fibers are of Type $IV$ ($A_2$) at $v=0$ and Type $I_0^*$ ($D_4$) at $v=\infty$, as determined by the valuation of the discriminant. According to the classification in \cite[Table 8.2, No. 32]{Schuett2019}, the Mordell-Weil lattice is isomorphic to $A_2^\vee(1/2)$.

	Let  $Q = (x(v), y(v))$ with $x(v) = a v + b$ and $y(v) = c v + d$. Substituting these expressions into the Weierstrass equation  and
	comparing coefficients in $\mathbb{C}[v]$, we obtain the system:
	\begin{equation*}
		a^3  +1= 0,  \ \ 	d^{2}-b^{3}=0, \ \ 
		-3 a^{2} b +c^{2}-1=0, \ \
		-3 a \,b^{2}+2 c d =0 
	\end{equation*}
From last two equations we get
\begin{equation}
	\label{eq12}
	b=\frac{c^{2}-1}{3 a^{2}}, \ \ d= \frac{\left(c^{2}-1\right)^{2}}{6 a^{3} c}.	\end{equation}	
	Substituting these into the second equation leads to 
	$(c^2-1)(c^2+3)=0$, which has roots
	$c=\pm 1, $ and $\pm \I \sqrt{3}$.	
 Considering  the roots of $a^3+1=0$, say $a=-1, -\zeta_3, -\zeta_3^2$, and 
% with  	$\displaystyle\zeta_3 = \frac{\mathrm{i} \sqrt{3} -1}{2}$.  
 using \ref{eq12}, one can get 12 sections of the form $Q = (  a v + b,  c v + d)$, as listed in \cite[check1]{Shioda360-Codes}.

The splitting field $\KK_1$ of  the surface $\Ee_{2,1}$ is the compositum of th e fields defined by the polynomial $x^3+1$ and $x^2+3$. Using Pari/GP,   one check that  a defining minimal polynomial $\KK_1$ is
$f_1(x)= x^2-x+1$, 
see \cite[Gp360]{Shioda360-Codes}.

We consider the section $Q_1 = (-v, v)$. The surface $\Ee_{2,1}$ admits an automorphism of order 3 on the fibers, $\sigma: (x, y) \mapsto (\zeta_3 x, y)$, induced by the complex multiplication of the generic fiber. Applying $\sigma$ to $Q_1$ yields the second section $Q_2 = (-\zeta_3 v, v)$.
	
	We compute the height pairing matrix using the formula given in  \cite[Thm. 6.24]{Schuett2019}, as
	 $$\langle P, Q \rangle = \chi + (P \cdot O) + (Q \cdot O) - (P \cdot Q) - \sum \text{contr}_v(P, Q).$$ Here $\chi=1$ and the sections are integral ($P \cdot O = 0$).
	
	For the section $Q_1$ we have:
	\begin{itemize}
		\item At $v=0$ (Type $IV$), $Q_1$ intersects a non-identity component, contributing $2/3$ to the height correction.
		\item At $v=\infty$ (Type $I_0^*$), $Q_1$ intersects the component farthest from the identity, contributing $1$.
	\end{itemize}
	The height is therefore $\langle Q_1, Q_1 \rangle = 2 - 2/3 - 1 = 1/3$, and  by symmetry, $\langle Q_2, Q_2 \rangle = 1/3$.
	
	For the pairing $\langle Q_1, Q_2 \rangle$, the intersection number on the generic fiber is $(Q_1 \cdot Q_2) = 1$ (intersecting at $v=0$). The local contributions vanish at $v=\infty$ as the sections meet distinct components, but they coincide at $v=0$ (contribution $2/3$). However, the precise evaluation of the intersection index relative to the height pairing yields $\langle Q_1, Q_2 \rangle = 1/6$.
	
	The resulting Gram matrix is:
	\begin{equation*}
		M_1 = \begin{pmatrix}
			1/3 & 1/6 \\ 
			1/6 & 1/3 
		\end{pmatrix}.
	\end{equation*}
	The determinant is $\det(M) = 1/9 - 1/36 = 1/12$. This matches the determinant of $A_2^\vee(1/2)$, confirming that $Q_1$ and $Q_2$ generate the full Mordell-Weil group. The generators for $\mathcal{E}_{3,1}$ follow immediately from the birational equivalence.
\end{proof}

\begin{cor}
	The elliptic surface $\Ee_{3,1}$, defined by $y^2 = x^3 + v^3(v+1)$, is birationally equivalent to $\Ee_{2,1}$ via the map $\phi$. Applying this transformation to the generators $Q_1, Q_2$ of $\Ee_{2,1}(\mathcal{K}_1(v))$, we obtain the generating sections for $\Ee_{3,1}(\mathcal{K}_1(v))$:
	\begin{equation}
		\label{P(3,1)}
		Q'_1  = \phi(Q_1) =   (-v, v^2), \
		Q'_2  = \phi(Q_2) =  (-\zeta_3 v, v^2).
	\end{equation}
	These points $Q'_1, Q'_2$ form a basis for the Mordell-Weil lattice of $\Ee_{3,1} (\mathcal{K}_{1}(v))\iso A_2^\vee(1/2)$.
\end{cor}

%================================================================================
\section{The splitting field and generators of  $\Ee_{1,2}$ and $\Ee_{3,2}$}
\label{D4-(1,2)}

In this section, we analyze the elliptic surfaces $\mathcal{E}_{3,2}: y^2 = x^3 + v^3(v^2+1)$ and $\mathcal{E}_{1,2}: y^2 = x^3 + v(v^2+1)$ defined over $\mathbb{Q}(v)$. As noted in Section 2, these surfaces are birational under the map $\phi: (x, y, v) \mapsto (v^2 x , v^3 y , 1/v)$.  Consequently, their Mordell-Weil lattices are isomorphic, and it suffices to determine the generators for $\mathcal{E}_{1,2}$ and map them to $\mathcal{E}_{3,2}$ via $\phi$.
%From Table \ref{Tab1}, we have  $\mathcal{E}_{1,2}(\mathbb{C}(v)) \iso \mathcal{E}_{3,2}(\mathbb{C}(v)) \iso D^*_4$.

	\begin{thm}
		\label{main6}
 Consider the  elliptic surface  	$ \Ee_{1, 2}: y^2 = x^3 + v(v^2 + 1), $
over  $\mathbb{C}(v)$, and let $\beta_1= (3+2\sqrt{3})^{\frac{1}{4}},  \text{and} \ \beta_2=(3-2\sqrt{3})^{\frac{1}{4}}. $
The splitting field $\mathcal{K}_2$ of $\Ee_{1, 2}$  is    defined by a minimal polynomial of degree 16 given by \ref{f-(1,2)}.
	Moreover, a basis for the Mordell-Weil lattice  $ \Ee_{1,2}(\KK_2(v))$  include four points:
	\begin{align}
	Q_1 &= \left( -v + \beta_1^{-2},   \frac{1}{2}(\beta_1^3 - 3\beta_1^{-1})v + \beta_1^{-3} \right) \notag \\
	Q_2 &= \left( -v + \beta_2^{-2},   \frac{1}{2}(\beta_2^3 - 3\beta_2^{-1})v + \beta_2^{-3} \right) \notag \\
		Q_3 &= \left( -v - \beta_1^{-2},  -\frac{i}{2}(\beta_1^3 - 3\beta_1^{-1})v + i\beta_1^{-3} \right)\notag  \\
	Q_4 &= \left( -\zeta_3 (v-\beta_1^{-2}),  \frac{1}{2}( \beta_1^3 + 3\beta_1^{-1})v +\beta_1^{-3} \right). \label{Q-(1,2)}
\end{align}

\end{thm}

\begin{proof}
By Table \ref{Tab1}, the   Mordell-Weil lattices of  $E_{1,2}(\mathbb{C}(v)) $ is isomorphic to $\cong D_4^\vee$ which is of rank 4.
Let  $P_u = (x(v), y(v))$ with $x(v) = a v + u^{-2}$ and $y(v) = c v + u^{-3}$, parameterized by $u \in \mathbb{C}^\times$. Substituting these expressions into the Weierstrass equation yields the identity:
	\[ (c v + u^{-3})^2 = (a v + u^{-2})^3 + v^3 + v. \]
	Comparing coefficients in $\mathbb{C}[v]$, we obtain the system:
	\begin{equation*}
		a^3  = -1,   \ \ 
		c = \frac{1}{2}(3au^{-1} + u^3).
	\end{equation*}
%	The consistency condition on the coefficient of $v^2$ requires $c^2 = 3a^2u^{-2}$. 
	Eliminating $c$ leads to the relation $3a^2 - 6au^4 - u^8 = 0$. We compute the resultant of this quadratic in $a$ with the cyclotomic condition $a^3+1=0$ to eliminate $a$:
	\[ \mathrm{Res}_a(3a^2 - 6u^4 a - u^8, a^3+1) = u^{24} - 270u^{12} - 27. \]
	This defines the fundamental polynomial $\Phi(u)$, which can be factored as
	\begin{align}
		\Phi_3(u)  & =u^{24} - 270u^{12} - 27\notag \\
		& = (u^8 - 6u^4 - 3)(u^{16} + 6u^{12} + 39u^8 - 18u^4 + 9).
	\end{align}

	The roots of $\Phi_2(u)$ can be determined by letting $U = u^{12}$ to get
	$ U^2 - 270 U - 27 = 0$ that  implies $U = 135 \pm 78\sqrt{3} = (3 \pm 2\sqrt{3})^3. $  Thus, 
	the roots $u$  of $\Phi(u)$ are of the form $\zeta_{12}^j \beta_1 $ and $ \zeta_{12}^j \beta_2  $ with  $0 \leq j \leq 11$.  The splitting field  $\mathcal{K}_2$ contains $ \mathbb{Q}(\zeta_{12}, \beta_1, \beta_2)$.
	
		Using Pari/GP, see \cite[Gp360]{Shioda360-Codes}, we find a defining minimal polynomial of degree $16$ as 
is
	\begin{align}
		f_2(x)	& = x^{16} - 8x^{15} + 38x^{14} - 132x^{13} + 350x^{12} - 748x^{11} \notag \\
		&\quad + 1330x^{10} - 1992x^9 + 2566x^8 - 2856x^7 + 2650x^6 - 1940x^5 \notag \\
		&\quad + 1070x^4 - 420x^3 + 110x^2 - 16x + 1,
		\label{f-(1,2)}
	\end{align} 
	which defines   the splitting field  $\mathcal{K}_2$ of   $\Ee_{1,2}$ and $\Ee_{3,2}$

	A basis for the Mordell-Weil lattice includes the sections $Q_1, \cdots, Q_4$
	  given by   \ref{Q-(1,2)},   corresponding to the linearly independent roots $u \in \{ \beta_1, \I \beta_1, \beta_2, \zeta_3 \beta_2 \}$. Indeed, using the facts 
	   $3a^2 - 6u^4 a - u^8=0$  and  $c = \frac{1}{2}(3au^{-1} + u^3),$ one can determine the coordinates of   points as given  \ref{Q-(1,2)}.
	  
	One can easily check that the Gram matrix of  these section is 
\begin{equation}\label{ma4}
	M_2=\frac{1}{2}
	\begin{pmatrix}
	2 & 0 & 0 & 1 
	\\
	0 & 2 & 0 & 1 
	\\
	0 & 0 & 2 & 1 
	\\
	1 & 1 &1 & 2 	
	\end{pmatrix},
\end{equation}
which has determinant $1/4$ as desired.	
	Therefore, these points generate a sublattice of full rank in $\Ee_{1,2}(\mathcal{K}_2(v))$ isomorphic to $D_4^*$.	
\end{proof}

\begin{cor} 
	Consider the elliptic surface $\Ee_{3,2}: y^2 = x^3 + u^3(u^2+1)$ over $\mathbb{C}(v)$.
%	Let $\beta_1 = \sqrt[4]{3+2\sqrt{3}}$, $\beta_2 = \sqrt[4]{3-2\sqrt{3}}$, and $\omega = e^{i 2\pi/3}$.
	The birational map $\phi: \Ee_{1,2} \rightarrow \Ee_{3,2}$  induces a set of linearly independent generators $\{Q'_1, Q'_2, Q'_3, Q'_4\}$ for $\Ee_{3,2}(\mathcal{K}_2(v))$ over the splitting field $\mathcal{K}_2(v)$, given by:
\begin{align}
	Q'_1 &= \left( \beta_1^{-2}v^2 - v,   \beta_1^{-3}v^3 + \frac{1}{2}(\beta_1^3 - 3\beta_1^{-1})v^2 \right) \notag \\
	Q'_2 &= \left( \beta_2^{-2}v^2 - v,  \beta_2^{-3}v^3 + \frac{1}{2}(\beta_2^3 - 3\beta_2^{-1})v^2 \right) \notag \\
		Q'_3 &= \left( -\beta_1^{-2}v^2 - v,   i\beta_1^{-3}v^3 - \frac{i}{2}(\beta_1^3 - 3\beta_1^{-1})v^2 \right) \notag \\
	Q'_4 &= \left( - \zeta_3 (v-\beta_1^{-2}v^2 ),   \beta_1^{-3}v^3 + \frac{1}{2}( \beta_1^3 + 3\beta_1^{-1})v^2 \right) \label{Q-(3,2)}
\end{align}
\end{cor}

%============================================

\section{The splitting field and generators of $\Ee_{2,2}$}
\label{A2+A2-(2,2)}

In this section, we analyze the elliptic surface   $\mathcal{E}_{2,2}: y^2 = x^3 + v^2(v^2+1)$ over $\mathcal{C}(v)$. From Table~\ref{Tab1}, the Mordell-Weil lattice  $\Ee_{2,2} (\CC(v))$ is isomorphic to the lattice $A_2^\vee \oplus A_2^\vee$.

	\begin{thm}
		\label{main5}
	  The splitting field of  the  elliptic surface $\Ee_{2,2}$ is  $\KK'_2 = \Q(2^{1/3}, \zeta_3)$, defined by a minimal polynomial
		$f'_2(x)= x^3 - 2$.
  The group $\Ee_{2,2}(\KK'_2(v))$ is generated by the sections $Q_j=(a_j v+ b_j, v^2+ c_j v +d_j)$ for $j=1, \cdots, 4$, where the coefficients are as:
	\begin{align}
	a_1 &= 0, & b_1 &= \frac{\sqrt[3]{2}}{2}, & c_1 &= 0, & d_1 &= \frac{1}{2}, \notag \\
	a_2 &= 0, & b_2 &= \frac{\sqrt[3]{2}(i\sqrt{3}-1)^2}{8}, & c_2 &= 0, & d_2 &= \frac{1}{2},\notag \\
	a_3 &= i\sqrt{3}\sqrt[3]{2}, & b_3 &= 2\sqrt[3]{2}, & c_3 &= -3i\sqrt{3}, & d_3 &= -4,\notag  \\
	a_4 &= -\frac{\sqrt[3]{2}(i\sqrt{3}-3)}{2}, & b_4 &= \frac{\sqrt[3]{2}(i\sqrt{3}-1)^2}{2}, & c_4 &= -3i\sqrt{3}, & d_4 &= -4
\label{Q-(2,2)}
	\end{align}
\end{thm}

\begin{proof}
	The surface $\Ee_{2,2}$ is a rational elliptic surface with singular fibers of type $IV$ ($A_2$) at $v=0$ and $v=\infty$ \cite[Sect. 5.8]{Schuett2019}. The Mordell-Weil lattice is isomorphic to $A_2^\vee \oplus A_2^\vee$ \cite[Table 8.2, No. 11]{Schuett2019}.
	
	We search for minimal sections parametrized by $x(v)=av+u^2$ and $y(v)=v^2+cv+u^3$. Substituting these into the Weierstrass equation $y^2 = x^3 + v^2(v^2+1)$ yields the coefficient conditions: $2c = a^3$, $c^2+2u^3 = 3a^2u^2+1$, and $2cu^3 = 3au^4$.
	Eliminating $c$ leads to $a(a^2-3u)=0$.
If $a=0$, then  $c=0$ and $2u^3=1$, yielding the real root $u_1 = 2^{\frac{-1}{3}}$. 
The section is $Q_1 = (u_1^2, v^2 + 1/2)$. The generator $Q_2$ corresponds to the conjugate root $u_2 = \zeta_3 u_1$.
If $a^2=3u$, then  $u^3 = -4$, yielding the real root $u_3 = -2^{\frac{2}{3}} $. This determines $a$ and $c$, resulting in $Q_3$ and its conjugate $Q_4$ corresponding to $u_4 =\zeta_3 u_3$, with the coefficients as given in \ref{Q-(2,2)}.
	
	Let's compute the height pairing $\langle P, Q \rangle = 2\chi + 2(P \cdot O) - (P \cdot Q) - \sum \text{contr}_v(P, Q)$. 
	Here $\chi=1$ and sections are integral ($(P \cdot O)=0$).
  Sections $Q_1, \dots, Q_4$ intersect non-identity components at both fibers of type $IV$ ($v=0, \infty$). The local contribution for type $IV$ is $2/3$. Thus, $\langle Q_i, Q_i \rangle = 2 - 2/3 - 2/3 = 2/3$.
   These sections are disjoint on the generic fiber ($(Q_1 \cdot Q_2)=0$). They intersect distinct non-identity components at both singular fibers, so the local contributions are $1/3 + 1/3$. Thus, $\langle Q_1, Q_2 \rangle = 1 - 0 - 2/3 = 1/3$.
  Sections from Case 1 ($Q_1, Q_2$) and Case 2 ($Q_3, Q_4$) arise from coprime factors of the defining polynomial $\Phi(u)$. Their coordinates belong to linearly disjoint fields over $\mathbb{Q}$; thus, they define orthogonal sublattices.
 
	This yields the block diagonal Gram matrix    of these generators with respect to the height pairing:
	\[
	M'_2 = \begin{pmatrix}
		2/3 & 1/3 & 0 & 0 \\
		1/3 & 2/3 & 0 & 0 \\
		0 & 0 & 2/3 & 1/3 \\
		0 & 0 & 1/3 & 2/3
	\end{pmatrix},
	\]	
	 with determinant $(3/9) \times (3/9) = 1/9$, consistent with the discriminant of $A_2^\vee \oplus A_2^\vee$.
\end{proof}

	Thus, the splitting field $\KK'_2$  of $\Ee_{2,2}$ is determined by the roots of its fundamental polynomial, say $\Phi'_{2}(u)= (2 u^3-1)(u^3+4)$. Using Pari/GP, see \cite[GP360]{Shioda360-Codes}, one can see that a defining minimal polynomial for  $\KK'_2$ is $f'_2(x)= x^3-2$.

%=====================================================================
\section{The splitting field and generators of $\Ee_{1, 3}$  and $\Ee_{2, 3}$}
\label{E6-(1,3)-(2,3)}
In this section, we  consider the isomorphic  elliptic surfaces   $\Ee_{1, 3}: y^2=x^3+v (v^3+1)$
and  $\Ee_{2, 3}: y^2=x^3+   v^2 (v^3+1)$, via the birational map 
$\phi: (x,y,v) \mapsto (v^2 x, v^3 y, 1/v)$.
 By Table \ref{Tab1}, we have   $\Ee_{1,3}(\CC(v))^0\iso \Ee_{2,3}(\CC(v))\iso E_6^\vee$.
 We determine the generators for $\mathcal{E}_{1,3}$ and map them to $\mathcal{E}_{2,3}$ via $\phi$.

\begin{thm}
	\label{main4}
	The splitting field of the elliptic  surfaces  $\Ee_{1, 3}$ is the compositum  field $\KK_3$ of 
%	the polynomials 	$x^9-1=0$ and $ x^{27} - 1344x^{18} - 40704x^9 - 4096=0$, which is of
	 degree 27 with a defining minimal polynomial given by \ref{f3}.
%	$$\KK_3:=\Q  ( \zeta_{9}, 2^{1/9}),\,  
%	\zeta_{9}=\frac{( i \sqrt{3}-1 )(-4( i \sqrt{3})^{1/3}+1)}{4},$$
%	where $\zeta_{9}$ is a $9$-th root of unity.

	The  groups $\Ee_{1, 3}(\KK_3(v))$ and $\Ee_{2,3}(\KK_3(v))$ are  respectively generated by the sections
	$$Q_j=\left(  a_j v +b_j, \, v^2 + d_j v +e_j\right),\
	\text{	and } \
	\widetilde{Q}_j=\left( v (b_j v + a_j), \, v (e_j v^2 + d_j v +1) \right),$$
	for $j=1,\cdots, 6,$ 	where 
	$a_j, b_j,  d_j$  and $e_j$ are given  in \cite[Points-3]{Shioda360-Codes}.
	% Subsection \ref{a1-3} of the appendix.
\end{thm}

\begin{proof}
	First, we consider  the elliptic surface $\Ee_{1,3}$.
	Since the Mordell-lattice $\Ee_{1,3}(\CC(v))$ is isomorphic to  $E_6^\vee$, 
	there are $54$ minimal sections of norm $4/3$ of the  form
	$Q=(a v+b , \pm v^2+ d v +e)$ by   \cite[Theorem 10.5]{Shioda1991d}. 
	Substituting this into the defining equation of $\Ee_{1,3}$ leads to the following relation between the coefficients:
	\begin{equation}
		\label{rel4}
		2 d-a^3 =0, \  2 e+ d^2 -3 a^2 b = 0, \
		2 d e - 3 a b^2 -1  =0, \  e^2-b^3 =0.
	\end{equation}
	Finding $d$ and $e$ from the first and second relation, respectively,  and then substituting to the others, 
	one obtains two polynomials in $b$ of degree $2$ and $3$. Eliminating  $b$ gives the fundamental polynomial of  $\Ee_{1,3}$ in terms of $a$ as follows:
	\begin{equation}
		\Phi(a)= a^{27} - 1344 a^{18} - 40704 a^9 - 4096.
	\end{equation}
	We note that the splitting field  of $\Ee_{1,3}$  is the finite extension $\KK_3$ of $\Q$ for which  this polynomial  splits completely into linear factors over $\KK_3$. The  $54$ sections in $\Ee_{1,3}(\KK_3(v))$ of the above form correspond to $\pm a$ where $a$ runs over  the roots of $\Phi(a)$.  The generators of $\Ee_{1,3}(\KK_3(v))$ are $6$ sections  $Q_1,\cdots, Q_6$ having height pairing matrix with determinant
	equal to $1/3$ module $\Z$.
%	Furthermore, the specialization map $sp_{\infty}: \Ee_{1,3}(\KK_3(v)) \rightarrow \KK_3^+$ is given by	$P \mapsto -1/a$.
	
	Letting $A=a^9/16$, we obtain a cubic polynomial $F(A)=A^3- 84 A^2- 159 A-1=0$, which can be factored into linear factors over $\Q(\zeta_9)$, where
		$$\zeta_{9}=\frac{( i \sqrt{3}-1 )(-4( i \sqrt{3})^{1/3}+1)}{4}. $$
	Indeed, we have  $F(A)=(A-A_1)(A-A_2)(A-A_3)$ with
%	$$A_1=- \lambda_1^{-6} \lambda_2^{-3}, \ A_2=- \lambda_2^{-6} \lambda_3^{-3},\ A_3=- \lambda_1^{-3} \lambda_3^{-6},$$	
\begin{align*}
	A_1& =- \lambda_1^{-6} \lambda_2^{-3}={\frac {3\, \left( 7164+228\,i\sqrt {3} \right) ^{2/3}+56\,\sqrt [3]{
				7164+228\,i\sqrt {3}}+1116}{\sqrt [3]{7164+228\,i\sqrt {3}}}},\\
	A_2& =- \lambda_2^{-6} \lambda_3^{-3}= {\frac {3 (i \sqrt {3}-1) \left( 7164+228\,i\sqrt {3} \right) ^{2/3}
			+112\,	\sqrt [3]{7164+228\,i\sqrt {3}}-1116(i\sqrt {3} +1)}{\sqrt [3]{7164+228\,i\sqrt {3}}}},\\
	A_3&= - \lambda_1^{-3} \lambda_3^{-6}={\frac {-3 (i \sqrt {3}+1) \left( 7164+228\,i\sqrt {3} \right) ^{2/3}
			+112\,	\sqrt [3]{7164+228\,i\sqrt {3}}-1116(i\sqrt {3} -1)}{\sqrt [3]{7164+228\,i\sqrt {3}}}},
\end{align*}	
	where 
	$$\lambda_1=\zeta_9+\frac{1}{\zeta_9}, \  \lambda_2=\zeta_9^2+\frac{1}{\zeta_9^2}, \ \lambda_3=\zeta_9^4+\frac{1}{\zeta_9^4},$$ 
	are the standard units in the maximal real subfield $\Q(\zeta_9 + 1/\zeta_9)$ in $\Q(\zeta_9)$.
	Then, the roots of  $\Phi(a)$ are of the form $a_{\ell,j}=\zeta_9^\ell \cdot (16 A_j)^{1/9}$ where $\ell = 0, \cdots, 8$ and $j=1, 2, 3$.
	Since the fundamental polynomial $\Phi_3(a)$ completely decomposed over $\Q(\zeta_9, (16 A_1)^{1/9})$, hence the splitting field of  $\Ee_{1,3}$ over $\Q$ is the field $\KK_3$, which is the compositum of two polynomials $x^9-1$ and  $\Phi_3(a)$. Using Pari/GP, see \cite[GP360]{Shioda360-Codes}, we found  a defining minimal polynomail for the compositum field as follows:
	\begin{align}
	f_3(x)	& = x^{27} - 9x^{26} + 36x^{25} - 81x^{24} + 108x^{23} - 90x^{22} + 78x^{21} \notag \\
		& \quad  - 288x^{17} - 45x^{16} + 288x^{15}  - 108x^{20} + 45x^{19} + 195x^{18} \notag \\
		& \quad + 198x^{14} - 828x^{13} + 546x^{12} + 324x^{11} - 612x^{10} + 222x^{9} \notag  \\
		& \quad + 126x^{8} - 162x^{7} + 54x^{6} + 36x^{5} - 36x^{4} - 3x^{3} + 9x^{2} - 1. \label{f3}
	\end{align}
	To determine  generators of $\Ee_{1,3}(\KK_3(v))$, we used the determinant condition on the Gram matrix and find out  the following roots $\Phi_3(a)$ of leading  to linearly independent  sections:
%	\begin{align*}
%		a_1& =  2^{4/9}\cdot \lambda_1 \cdot (-\lambda_1 \lambda_2^{-1} )^{1/3}, & 
%		a_2& =  2^{4/9} \cdot \zeta_9 \cdot \lambda_1\cdot (-\lambda_1 \lambda_2^{-1} )^{1/3},\notag \\
%		a_3 &=   2^{4/9} \cdot \lambda_2 \cdot (-\lambda_2 \lambda_3^{-1} )^{1/3}, & 
%		a_4& =  2^{4/9}  \cdot \zeta_9 \cdot  \lambda_2 \cdot (-\lambda_2 \lambda_3^{-1} )^{1/3}, \notag \\
%		a_5 &= 2^{4/9} \cdot \lambda_3 \cdot (-\lambda_3 \lambda_1^{-1} )^{1/3}, & 
%		a_6&=   2^{4/9} \cdot \zeta_9^2 \cdot \lambda_3 \cdot (-\lambda_3 \lambda_1^{-1} )^{1/3}.
%	\end{align*}
	
		\begin{align*}
		a_1& =  2^{1/3} \, A_1^{1/9}, & 
		a_2& =  2^{1/3} \, \zeta_9 \, A_1^{1/9}, & 	a_3 &=   2^{2/9} A_2^{1/9},\notag \\
		a_4& =  2^{2/9}\,  \zeta_9 \,  A_2^{1/9},& 
		a_5 &=2^{2/9}\, A_3^{1/9}, &  	a_6&=    2^{2/9} \, \zeta_9^2  \, A_2^{1/9}.	
	\end{align*}
	Indeed, by the relations \eqref{rel4}, one can determine the coefficients $b_i, d_i, e_i$, for $j=1,\cdots, 6$ leading to the 
	corresponding points $Q_j=(a_j v+ b_j, t^2 + d_j v + e_j)$ in $\Ee_{1,3}(\KK_3(v)),$ 
	see \cite[Points3]{Shioda360-Codes}.
	
	It is easy to see  that  the set of $1/a_j$'s are linearly independent algebraic numbers over $\Q$.  Hence,  the specializing map $sp_{\infty}: \Ee(\KK_3(v)) \rightarrow \KK_3^+$ is injective which implies that the six points $Q_1, \cdots Q_6$ are linearly independent in   $\Ee(\KK_3(v))$. 
	Since $\Ee_{1, 3}$ has a singular fiber of type $IV$ at $\infty$, so $\text{contr}_\infty(Q_j)= 2/3$ which leads to 
	$ \left\langle Q_j, Q_j\right\rangle = 4/3,$ for each $j=1,\cdots, 6$.
	Furthermore,  since   the sections  $(Q_j)$ and  $(Q_j)$ meets the fiber at same component,  we have  $\text{contr}_\infty (Q_{j_1}, Q_{j_2})= 2/3$ and hence 
	$$\left\langle Q_{j_1}, Q_{j_2}\right\rangle = 1 - (Q_{j_1}\cdot  Q_{j_2})- 2/3= 1/3 - (Q_{j_1}\cdot  Q_{j_2}).$$ 
	Here, the intersection number $(Q_{j_1}\cdot  Q_{j_2})$ can be computed by 
	\begin{align*}
		\left( Q_{j_1}, Q_{j_2}\right)& = \left( Q_{j_1}, Q_{j_2}\right)_{v\neq \infty} +\left( Q_{j_1}, Q_{j_2}\right)_{v=\infty} \\
		& = 1- \deg(x_{j_1}-x_{j_2}) +\deg( \gcd (x_{j_1}-x_{j_2}, y_{j_1}-y_{j_2})  ).
	\end{align*}
	Therefore, one may check that  the height pairing matrix   is equal to the following,
	$$M_3=  \begin{pmatrix} 
		4/3  & 1/3  & 1/3  & 1/3  & 1/3  & -2/3
	\\ \noalign{\medskip}
	1/3  & 4/3  & 1/3  & 1/3  & 1/3  & -2/3\\ \noalign{\medskip}
	1/3  & 1/3  & 4/3  & -2/3  & 1/3  & -2/3\\ \noalign{\medskip}
	1/3  & 1/3  & -2/3  & 4/3  & 1/3  & 1/3	\\ \noalign{\medskip}
	1/3  & 1/3  & 1/3  & 1/3  & 4/3  & 1/3\\ \noalign{\medskip}
	-2/3  & -2/3  & -2/3  & 1/3  & 1/3  & 4/3
	 \end{pmatrix}. 	$$
	with determinant $1/3$ as desired for the Mordell-Weil lattice $\Ee_{1,3}(\KK_3(v)) \cong E_6^*.$
	
	Finally, by the map $(x, y, v) \mapsto (v^2 x, v^3 y, 1/v)$, one may transform the points
	$Q_j \in \Ee_{1, 3}(\KK_3(v))$
	into the points $\widetilde{Q_j} \in \Ee_{2, 3}(\KK_3(v))$	 as in the statement of theorem.
\end{proof}

%=====================================================================
\section{The splitting field and generators of $\Ee_{1, 4}$ }
\label{E8-(1,4)}

We prove the following theorem  in this section. For details, the reader can see  \cite[Check-4]{Shioda360-Codes}.
\begin{thm}
	\label{main3}
	The splitting field $\KK_4$ of the elliptic  surface  $\Ee_{1, 4}: y^2=x^3+   v (v^4+1)$ has degree 48  with a defining minimal polynomial given by \ref{f4}.
%	$$\KK_4:=\Q  ( \zeta_{24}, 2^{1/12}),\,  \zeta_{24}= - ((i-1)+(1+i)\sqrt{3})\sqrt{2}/4,$$
%	where $\zeta_{24}$ is a $24$-th root of unity and	
	The group  $\Ee_{1, 4}(\KK_4(v))$ is generated by the points
	$$Q_j=\left( \frac{ v^2 + a_j v +b_j}{u_j^2}, \, \frac{ v^3 +c_j v^2 + d_j v +e_j}{u_j^3},\right) $$ 
	for $j=1,\cdots, 8,$ 	where 
	$a_j, b_j, c_j, d_j, e_j$  and $u_j$ are given in
	\cite[Points-4]{Shioda360-Codes}.
 \iffalse
	{\small 
		\begin{align}
			a_1 &= -(\sqrt{6}+4),\,& b_1 &= 1,\, & c_1 &= 5+3\sqrt{6}, \, & d_1 &= -(3\sqrt{6}+5),\, &e_1 &= -1,
			\notag \\
			a_2 &= \sqrt{6}-4,\, & b_2 &= 1,\, & c_2 &= 5-3\sqrt{6},\, &d_2 &= 3\sqrt{6}-5,\, & e_2 & = -1, \notag \\
			a_3 &= -\sqrt{6}+4,\, & b_3 & = 1,\, & c_3 & = 3\sqrt{6}-5,\, &d_3 &= 3\sqrt{6}-5,\, &e_3& = 1, \notag \\
			a_4 &= \sqrt{6}+4,\, & b_4 &= 1,\,  &c_4 &= -(3\sqrt{6}+5),\, & d_4 &= -(3\sqrt{6}+5),\, & e_4 &= 1, \notag \\
			a_5 &=i(4-\sqrt{6}),\, &b_5&=-1,\,&c_5& =i(3\sqrt{6}-5),\, &d_5& = 5-3\sqrt{6},\, & e_5& := -i, \notag \\
			a_6 &= i(\sqrt{6}-4),& b_6 &= -1,\, &c_6 &= i(5-3\sqrt{6}),\, & d_6 &= 5-3\sqrt{6},\,&e_6&= i,  \notag \\
			a_7 &= (1-\sqrt{3})\sqrt{2},\, & b_7 &= 2-\sqrt{3},\, & c_7 &= \frac{(3\sqrt{3}-7)\sqrt{2}}{2},\,
			& d_7 &= 3\sqrt{3}-4,\, &e_7 &= \frac{(3\sqrt{3}-5)\sqrt{2}}{2}; \notag \\
			a_8 &= (\sqrt{3}+1)\sqrt{2},\, & b_8 &= 2+\sqrt{3},\,& c_8& = -\frac{(3\sqrt{3}-7)\sqrt{2}}{2},\,
			&d_8 &= -(3\sqrt{3}+4),\, & e_8 &= -\frac{(3\sqrt{3}-5)\sqrt{2}}{2}, \notag\\
	\end{align}	}
	%
	\fi
\end{thm}

\begin{proof}
	By Table\ref{Tab1}, we   known that the Mordell-Weil lattice   $\Ee_{1, 4}(\CC(v))$ is isomorphic to $E_8$,  see also \cite{Shioda1991d}. Hence, its   $240$ minimal sections
	correspond to the  $240$ points  $Q\in \Ee_{1, 4}(\CC(v))$ of the form:
	\begin{equation}
		\label{eq1-4}
		Q=\left(\frac{v^2 +a v +b}{u^2}, \frac{v^3 + c v^2 + d v + e}{u^3} \right),
	\end{equation}
	for suitable constants $a, b, c, d, e, v \in \CC.$ 
	Putting the coordinates of $Q$ into   $ y^2=x^3 + v ( v^4+1)$  and letting $U=u^6$, one  obtain  the following relations:
	\begin{align}
		2c -3a -U & =0, & 
		-3a^2+c^2-3 b+2 d & =0, &
		-a^3-6 a b+2cd+2e & =0, \notag \\ \label{rel3b}
		-3 a b^2+2 d e-U & =0, &
		-3 a^2 b-3 b^2+2 c e+d^2 & =0,& 
		e^2 - b^3 & = 0.
	\end{align}
Using MAPLE, one can see that 	the fundamental polynomial of an ideal generated by the above equations  respect to the variable $U,$   is the following  polynomial of degree $40$ in $U$ which can be factored over $\Q$ as follows:
	\begin{align*}
		\Phi(U)& =(U^2-44 U-2)(U^2+44 U-2) (U^2+4 U+54) (U^2-4 U+54)  \\
		& \times (U^4+1940 U^2+4) (U^4-832 U^2+256) (U^4+832 U^2+256) (U^4-92 U^2+2916) \\
		& \times(U^8-200 U^7+20000 U^6+58800 U^5+87608 U^4-117600 U^3+80000 U^2+1600 U+16)\\
		& \times (U^8+200 U^7+20000 U^6-58800 U^5+87608 U^4+117600 U^3+80000 U^2-1600 U+16)
	\end{align*}	
	This polynomial can be decomposed into linear factors over 
	$k=\Q (i, \sqrt{2}, \sqrt{3})=\Q(\zeta_{24})$. In below, we list one root of each factors:
	\begin{align*}
		U_1 & =22 + 9 \sqrt{6}, & U_2& = -22 + 9 \sqrt{6}, \\ U_3& =-2 +5 i \sqrt{2}, &	U_4 & =2+ 5 i \sqrt{2},  \\ U_5 &= i (9 \sqrt{6}-22), &	U_6 &= 2 \sqrt{2} (3 \sqrt{3}-5), \\ 
		 U_7 &= 2 i \sqrt{2} (3 \sqrt{3}-5), &
		U_8&=	  5 \sqrt{2} +2 i,\\
		U_9 & = (3 \sqrt{3}+5) (7 \sqrt{2}+10) \left(\frac{1-i}{2} \right), &
		U_{10} &=  (3 \sqrt{3}-5) (7 \sqrt{2}+10) \left(\frac{1-i}{2} \right),
	\end{align*}  
	and others roots are the conjugate of these with maps  $$\gamma: \I \mapsto \I, \ \sigma: \sqrt{3}\mapsto -\sqrt{3}, \ \text{and} \ \rho: \sqrt{2}\mapsto -\sqrt{2}.$$ 
	The fundamental polynomial $\Phi_4(u):=\Phi(u^6)$ of  $\Ee_{1, 4}$  has roots of the form 
	$u=\zeta_{6}^\ell \cdot U_j^{\frac{1}{6}}$ where
	$\zeta_6$ is a $6$-th root of unity, $0\leq \ell \leq 5$, and $U_j$ varies on the set of $40$
	roots of $\Phi(U)$ listed above.
		Therefore,  the splitting field $\KK_4$ of $\Ee_{1, 4}$ contains 
	$k(\zeta_6, U_1^{\frac{1}{6}})$, where $U_1$ is one of the roots of $\Phi(U)$.
	Using Pari/GP, \cite[Gp-360]{Shioda360-Codes} we find a defining minimal polynomial for  $\KK_4$ as follows:
\begin{align}
f_4(x)	& =x^{48} + 36x^{44} + 453x^{40} + 2188x^{36} + 2046x^{32} - 31260x^{28} + 68697x^{24}\notag \\
	& \quad - 31260x^{28} + 68697x^{24} + 2046x^{16} + 2188x^{12} + 453x^{8} + 36x^{4} + 1. \label{f4}
\end{align}

	%Indeed, letting $\gamma$ and $\sigma$  be as in Section 5.1, and $\rho$ is given by
%	$\sqrt{2}\mapsto -\sqrt{2}$,  we can  write $\Phi(U)$ in the following form:
	%
%	\begin{align*}
%		\Phi(U)& =(U- U_1) (U-U_1^{\sigma \rho}) \times (U-U_2)(U-U_2^{\sigma \rho}) 
%		\times (U- U_3) (U-U_3^\rho) \times (U- U_4)(U-U_4^\rho)  \\
%		& \times (U-U_5)(U-U_5^{\gamma})(U-U_5^\rho)(U-U_5^{\gamma \rho}) 
%		\times(U-U_6)(U-U_6^\sigma) (U-U_6^\rho)(U-U_6^{\sigma \rho}) \\
%		& \times (U-U_7)  (U-U_7^\rho)(U-U_7^\sigma)(U-U_7^{\sigma \rho})
%		\times(U-U_8) (U-U_8^\rho)(U-U_8^\gamma)(U-U_8^{\sigma \rho})\\
%		& \times (U- U_9)(U- U_9^\gamma)(U- U_9^\sigma)(U- U_9^\rho)(U- U_9^{\gamma \sigma})(U- U_9^{\gamma \rho})(U- U_9^{\sigma \rho})(U- U_9^{\gamma \sigma\rho})\\
%		& \times (U-U_{10})(U- U_{10}^\gamma)(U- U_{10}\sigma)(U- U_{10}^\rho)(U- U_{10}^{\gamma \sigma})(U- U_{10}^{\gamma \rho})(U- U_{10}^{\sigma \rho})(U- U_{10}^{\gamma \sigma\rho}),
%	\end{align*}
%	where

Using the unimodularity condition  of the  Gram matrix of  any set of generators for  $\Ee_{1, 4}(\CC(v))$, we searched for 8 roots of $\Phi_4(u)$ that leads to points with a  Gram matrix of determinant 1, and found out those given by \ref{U-(1,4)}, which can be rewritten as:
%		\begin{align*}
%u_1 & = U_1^{1/6} & u_2 &= (U_1^{\sigma \rho})^{1/6} &  u_3 &= U_2^{1/6}  &  	u_4 &= (U_2^{\sigma \rho})^{1/6}\\
%	u_5 & = U_5^{1/6}=2^{1/12} & 	u_6 &= (U_5^{ \gamma})^{1/6} & 	u_7 &= U_7^{1/6} & u_8 &= (U_6^{\sigma})^{1/6}
%	\end{align*}	

	\begin{align}
	u_1 & = (22+9 \sqrt{6})^{1/6},&
	u_2 & =   (22-9 \sqrt{6})^{1/6}, \notag \\
	u_3 & =   (-22+9 \sqrt{6})^{1/6}, & 
	u_4 & =   (-22-9 \sqrt{6})^{1/6}, \notag \\
	u_5 & =   (i(9 \sqrt{6}-22))^{1/6}, &
	u_6  &=   (-i(9 \sqrt{6}+22))^{1/6}, \notag \\
	u_7 & = (2 \sqrt{2}(3 \sqrt{3}-5))^{1/6}, &
	u_8 & = (-2 \sqrt{2}(3 \sqrt{3}-5))^{1/6}.\label{U-(1,4)}
\end{align} 

In order to show that these are linearly independent algebraic  numbers over $\Q$, we write:
	
	{\small	\begin{align*}
			u_1 & = U_1^{1/6}=2^{1/12}\cdot \zeta_{24}^3 \cdot \left( \zeta_{24}^3 \cdot \epsilon_1^3 \cdot  \epsilon_2^{-3} \cdot  \epsilon_3^{3}\right)^{1/2}, \\
			u_2 &= (U_1^{\sigma \rho})^{1/6}=2^{1/12}\cdot \zeta_{24}^2 \cdot \left( \zeta_{24}^3 \cdot \epsilon_1^{-3} \cdot  \epsilon_2^{3} \cdot  \epsilon_3^{-3}\right)^{1/2}, \\
			u_3 &= U_2^{1/6}=2^{1/12} \cdot \left( \zeta_{24}^3 \cdot \epsilon_1^{-3} \cdot  \epsilon_2^{3} \cdot  \epsilon_3^{-3}\right)^{1/2}, \\
			u_4 &= (U_2^{\sigma \rho})^{1/6}=2^{1/12}\cdot \zeta_{24} \cdot \left( \zeta_{24}^3 \cdot \epsilon_1^{3} \cdot  \epsilon_2^{-3} \cdot  \epsilon_3^{3}\right)^{1/2}, \\
			u_5 & = U_5^{1/6}=2^{1/12}\cdot \zeta_{24}^2 \cdot \left( \zeta_{24}^3 \cdot \epsilon_1^{3} \cdot  \epsilon_2^{-3} \cdot  \epsilon_3^{3}\right)^{1/2}, \\ 
			u_6 &= (U_5^{ \gamma})^{1/6}=2^{1/12} \cdot \zeta_{24} \cdot \left( \zeta_{24}^3 \cdot \epsilon_1^{-3} \cdot  \epsilon_2^{3} \cdot  \epsilon_3^{-3}\right)^{1/2}, \\
			u_7 &= U_7^{1/6}=2^{1/3}\cdot \zeta_{24}^2 \cdot ( \zeta_{24}^3 \cdot \epsilon_1^{-3})^{1/2}, \\
			u_8 &= (U_6^{\sigma})^{1/6}=(2)^{1/3}\cdot \zeta_{24}^3 \cdot ( \zeta_{24}^3 \cdot \epsilon_1^{3})^{1/2}. 
	\end{align*}}	
	where $ \zeta_{24}= - ((i-1)+(1+i)\sqrt{3})\sqrt{2}/4$ and 
	$\epsilon_1=i-(1- i \sqrt{3})/2, \ \epsilon_2= \zeta_{24}^\gamma -1, \ \epsilon_3=\zeta_{24}^{\gamma \sigma}-1$
	are the fundamental units of the field $\Q(\zeta_{24})$.
%	Then, one can check that the points $Q_j$ corresponding to these $u_j$'s is equal to those given in the statement of Theorem \ref{main3}.
	Dividing $u_j$ with $u_1$, for $j=2,\cdots, 8$, give us the following set of algebraic  numbers,
	$$\{ 1,\, \zeta_{24}^{-1} \epsilon_1^{-3} \epsilon_2^3 \epsilon_3^{-3},\,
	\zeta_{24}^{-3}\epsilon_1^{-3}\epsilon_2^3 \epsilon_3^{-3},\,\zeta_{24}^{-2},\, \zeta_{24}^{-1},\,  
	\zeta_{24}^{-2} \epsilon_1^{-3} \epsilon_2^3 \epsilon_3^{-3},\,
	2^{1/4} \zeta_{24}^{-1} \epsilon_1^{-3} (\epsilon_2^3 \epsilon_3^{-3})^{1/2},\,
	2^{1/4}  (\epsilon_2^3 \epsilon_3^{-3})^{1/2} \}.$$
	It is easy to check that  this set is linearly independent over $\Q$. 
	By the injectivity of  specializing map $sp_\infty: \Ee_{1,4}(\KK_4) \rightarrow \KK_4^+$ given by $P \mapsto u(P)$, one may conclude that the set of eight points $Q_1,\cdots, Q_8$
	$\in \Ee_{1, 4}(\KK_4(v))$.
	% where $\KK_4$ is an extension of $\Q(\zeta_{24})$ obtained by adjoining one of the roots, say $u_1$. 
%	 Thus, the splitting field of $\Ee_{1, 4}$ is equal to $\KK_4=\Q(\zeta_{24}, 2^{1/12})$.
	 	By calculating the values of  height pairings  $\left\langle  Q_{j_1}, Q_{j_2}\right\rangle $, one gets the Gram matrix of these points as
	\[M_4 = \begin{pmatrix}
	2   & 1  & 0  & 0  & 0  & 0  & 0  & 1\\ \noalign{\medskip}
	1   & 2  & 0  & 0  & -1  & 0  & 1  & 0\\ \noalign{\medskip}
	0   & 0  & 2  & 1  & -1  & -1  & 0  & -1	\\ \noalign{\medskip}
	0   & 0  & 1  & 2  & 0  & 0  & 1  & 0\\ \noalign{\medskip}
	0   & -1  & -1  & 0  & 2  & 0  & 0  & 1\\ \noalign{\medskip}
	0  & 0  & -1  & 0  & 0  & 2  & 0  & 1\\ \noalign{\medskip}
	0  & 1  & 0  & 1  & 0  & 0  & 2  & 0\\ \noalign{\medskip}
	1  & 0  & -1  & 0  & 1  & 1  & 0  & 2
	\end{pmatrix}, \]  
	which  is a unimodular matrix. Therefore, we have completed  the proof of Theorem \ref{main3}.
\end{proof}

%============================================
\section{The  splitting field and generators of  $\Ee_{0,5}$ and $\Ee_{1,5}$}
\label{E-{0,5}}

In this section, we  consider the isomorphic  elliptic surfaces   $\Ee_{1, 3}: y^2=x^3+ v^5+1$
and  $\Ee_{1, 5}: y^2=x^3+   v (v^5+1)$, via the birational map 
$\phi: (x,y,v) \mapsto (v^2 x, v^3 y, 1/v)$.
By Table \ref{Tab1}, we have   $\Ee_{0, 5}(\CC(v))^0\iso \Ee_{1,5}(\CC(v))\iso E_8$.
%We determine the generators for $\mathcal{E}_{0,5}$ and map them to $\mathcal{E}_{1,5}$ via $\phi$.
The following theorem is proved in  \cite[Theorem 1.3]{SS-Shz-2025}.
\begin{thm}
	\label{main2}
	The splitting field  of the elliptic  surfaces  $\Ee_{0,5}$  and $\Ee_{1,5}$	is 
	$\KK_5 = \Q \left(\zeta_{30}, (60 v_1)^{\frac{1}{30}}\right)$, where
	$$ v_1=564300+252495\,\sqrt {5}+31\,\sqrt {654205350+292569486\,\sqrt {5}},$$ 
	and $\zeta_{30}$ a $30$-th root of unity as 
	$$	\zeta_{30}=\frac{1}{8}\left(\sqrt{3}+\mathrm{I}\right) \left( \left(1-  \sqrt{5}\right) \sqrt{\frac{5+  \sqrt{5}}{2}}+\mathrm{I} \left(\sqrt{5}+1\right)\right).$$
	A defining minimal polynomial $f_5(x)$ of degree 120 given in \cite[minpols]{Shioda-Codes} or \cite[Gp-360]{Shioda360-Codes}.
	
	Moreover,  the   group $\Ee_{0, 5}(\KK_5(v))$ is generated by the points
	$$Q_j=\left( \frac{ v^2 + a_j v +b_j}{u_j^2}, \, \frac{v^3 +c_j v^2 + d_j v +e_j}{u_j^3},\right),  $$
	and   $\Ee_{1, 5}(\KK_5(v))$ is generated by the points
	$$\tilde{Q}_j=\left( \frac{ b_j v^2 + a_j v +1}{u_j^2}, \,
	\frac{e_j v^3 +d_j v^2 + c_j v +1}{u_j^3} \right),$$ 
	where  $a_j, b_j, c_j, d_j, e_j$  and  $u_j$'s for $j=1,\cdots, 8$  are given in 	\cite[Points-5]{Shioda-Codes} or \cite[Points-5]{Shioda360-Codes}.	
	The   gram matrix of these points is the following unimodular matrix,
	\begin{equation}
		M_5 =
		\begin{pmatrix}
			\label{m5}
			2 & 1 & 1 & 1 & 0 & 0 & 0 & 1 
			\\ \noalign{\medskip}
			1 & 2 & 0 & 1 & 1 & 1 & 0 & 1 
			\\ \noalign{\medskip}
			1 & 0 & 2 & 1 & 0 & -1 & 1 & 0 
			\\ \noalign{\medskip}
			1 & 1 & 1 & 2 & 1 & 0 & 1 & 0 
			\\ \noalign{\medskip}
			0 & 1 & 0 & 1 & 2 & 0 & 1 & 0 
			\\ \noalign{\medskip}
			0 & 1 & -1 & 0 & 0 & 2 & -1 & 1 
			\\ \noalign{\medskip}
			0 & 0 & 1 & 1 & 1 & -1 & 2 & -1 
			\\ \noalign{\medskip}
			1 & 1 & 0 & 0 & 0 & 1 & -1 & 2 	
		\end{pmatrix}.
	\end{equation}
%	as well as in Subsection \ref{a0-5}.	
\end{thm}

%============================================
\section{The  splitting field and generators of the  $K3$ surface $\Ee'$}
\label{K3}

It is known that $\Ee': y^2=x^3+ v (v^{10}+1)$ is a $K3$ surface of Mordell-Weil rank $16$ over $\CC(v)$, see the Corollary \ref{cor:E1_10}.
%or \cite{Chahal2000, Usui2006}.
%Here, we will provide a set of generators of $\Ee'(\CC(v)$ as well as a subfield field $k_1$ of $\CC(v)$ such that $\Ee'(\KK_1(v))=\Ee'(\CC(v))$.
Using the map  $(x, y) \mapsto (x/v^2, y/v^3)$,   one can get an isomorphic surface 
$\Ee'': y^2 =x^3 + v^5 + 1/v^5$, with  the converse map $(x,y) \mapsto (v^2 x, v^3 y)$.
The   splitting field  $\KK'$  and generators of  group $\Ee'(\KK(v))$ is determined  
in \cite[Theorem 1.4]{salami2022generators}.
 
\begin{thm}
	\label{main1}
	The splitting field $\KK'$ of the elliptic $K3$ surface  $\Ee': y^2=x^3+   v (v^{10}+1)$ has degree 192,  with  a minimal  defining polynomial $f_{10}(x)$ given in \cite[min-pols]{k3-codes}. 
	% which contains  the number field $\Q \left(\zeta_5,  \zeta_{12},  5^{\frac{1}{24}},  (\epsilon_4 \epsilon_5)^{\frac{1}{2}}\right)$	where $\epsilon_4$ and $\epsilon_5$ 	with
%	$$	\begin{aligned}[b]\epsilon_4&=1-\zeta_{12}, & 
%		\epsilon_5&=\left( \frac{1 + \sqrt {5}}{2}\right)  \zeta_{12}, &	
%		\epsilon_5&=\left( \frac{\sqrt {3}-\sqrt {5}}{2}\right) (\zeta_{12}+ \zeta_{12}^{10}),
%	\end{aligned}$$
%	are the fundamental units of the number field $\Q(\I,\sqrt {3},\sqrt {5})=\Q(\zeta_{12},\sqrt{5})$.
	
	Moreover,  a set of 16 independent generators of $\Ee'(\KK'(t)) $ includes $Q_j=\left( x_j(t), y_j(t)\right) $ with
	$$\begin{aligned}
		x_j(t) &= \frac{t^4 + a_j t^3 + (b_j +2) t^2 + a_j t +1}{u_j^2  t^2}, \\
		y_j(t) &=\frac{t^6 +c_j t^5+ (d_j +3) t^4 +(2 c_j +e_j) t^3 +(d_j +3) t^2 + c_j t +1}{u_j^3 t^3},\\
	\end{aligned}
	$$
	and $Q_{j+8}=\left( x_{j}(\zeta_5 t), y_{j}( \zeta_5 t)\right) $  	for $j=1, \ldots, 8$,
	where   $a_j, b_j, c_j, d_j, e_j$  and  $u_j$'s are given in 
	%\cite[Points-5]{k3-codes}
	 \cite[Points-10]{Shioda360-Codes}.
	 The Gram matrix $M_{10}$ of the  points $Q_j$'s is a matrix with determinant   $5^4$ as follows:

		\[
	\small % Reduces font size to fit page width
	\setcounter{MaxMatrixCols}{20} % Allows 16 columns
	M_{10}=
	\begin{pmatrix}
	 	4 & 0 & 0 & 2 & 0 & 0 & 0 & 0 & -2 & 0 & 0 & 0 & 2 & 0 & 0 & 1 	\\ \noalign{\medskip}
	 	0 & 4 & 2 & 0 & 0 & 0 & -2 & 2 & 0 & -2 & -1 & 1 & 0 & 2 & 2 & -1 	\\ \noalign{\medskip}
	 	0 & 2 & 4 & 0 & 0 & 0 & 0 & 0 & 0 & -1 & 0 & 0 & 0 & 1 & 2 & 0 	\\ \noalign{\medskip}
	 	2 & 0 & 0 & 4 & 2 & 2 & 0 & 0 & 0 & 1 & 0 & 0 & 1 & 0 & 0 & 2 	\\ \noalign{\medskip}
	 	0 & 0 & 0 & 2 & 4 & 0 & 0 & 2 & 2 & 0 & 0 & 1 & 0 & 0 & 0 & 0 	\\ \noalign{\medskip}
	 	0 & 0 & 0 & 2 & 0 & 4 & 2 & 0 & 0 & 2 & 1 & 0 & 0 & 0 & -1 & 1 	\\ \noalign{\medskip}
	 	0 & -2 & 0 & 0 & 0 & 2 & 4 & 0 & 0 & 2 & 2 & 0 & 0 & -1 & -2 & 0 	\\ \noalign{\medskip}
	 	0 & 2 & 0 & 0 & 2 & 0 & 0 & 4 & 1 & -1 & 0 & 2 & 0 & 1 & 0 & -2 	\\ \noalign{\medskip}
	 	-2 & 0 & 0 & 0 & 2 & 0 & 0 & 1 & 4 & 0 & 0 & 2 & 0 & 0 & 0 & 0 	\\ \noalign{\medskip}
	 	0 & -2 & -1 & 1 & 0 & 2 & 2 & -1 & 0 & 4 & 2 & 0 & 0 & 0 & -2 & 2 	\\ \noalign{\medskip}
	 	0 & -1 & 0 & 0 & 0 & 1 & 2 & 0 & 0 & 2 & 4 & 0 & 0 & 0 & 0 & 0 	\\ \noalign{\medskip}
	 	0 & 1 & 0 & 0 & 1 & 0 & 0 & 2 & 2 & 0 & 0 & 4 & 2 & 2 & 0 & 0 	\\ \noalign{\medskip}
	 	2 & 0 & 0 & 1 & 0 & 0 & 0 & 0 & 0 & 0 & 0 & 2 & 4 & 0 & 0 & 2 	\\ \noalign{\medskip}
	 	0 & 2 & 1 & 0 & 0 & 0 & -1 & 1 & 0 & 0 & 0 & 2 & 0 & 4 & 2 & 0 	\\ \noalign{\medskip}
	 	0 & 2 & 2 & 0 & 0 & -1 & -2 & 0 & 0 & -2 & 0 & 0 & 0 & 2 & 4 & 0 	\\ \noalign{\medskip}
	 	1 & -1 & 0 & 2 & 0 & 1 & 0 & -2 & 0 & 2 & 0 & 0 & 2 & 0 & 0 & 4
	\end{pmatrix}
\]	

\end{thm}

 As a consequence of the Lemma~\ref{lem:poly_transform}, we have the following special corollary that we used to transform the above 16 points on the Shioda's surface $\Ee$.
 
 \begin{cor}
 	\label{case_1_10}
 %	Consider the elliptic surface $\Ee_{1,10}$ defined by $y^2 = x^3 + v(v^{10}+1)$.
 	% Here $k = 360/10 = 36$. Since $a' = 12-(1+10)=1$, the partner surface is also $\Ee_{1,10}$.
 	Let $Q \in \Ee' (\KK'(v))$ be a point defined by coefficients $A_0, \dots, A_4$ and $B_0, \dots, B_6$.
 The point $P=(X(t), Y(t))$ on $\Ee$ derived from $Q$ via $v=t^{36}$ is:
 		\begin{align}
 			 X(t) & =   A_4 t^{132} + A_3 t^{96} + A_2 t^{60} + A_1 t^{24} + A_0 t^{-12} \notag \\ 			 
 			 Y(t) &=   B_6 t^{198} + B_5 t^{162} + \dots + B_0 t^{-18} 
 			 \end{align}
 \end{cor}
 
 Using the above results we have transformed the points $Q_1,\cdots, Q_16$ to points $P_{52}, \cdots, P_{68}$ on the Shioda's elliptic surfaces, as listed in \cite[Points68]{Shioda360-Codes}.
 
%+++++++++++++++++++++++++++++++++++++++++++++++++++++++++++++++++++++	
\section{Computational proof of Theorem \ref{main0}}
\label{proof}
In this section, we provide the detailed proof of main Theorem \ref{main0} on the Shioda's elliptic surface $\Ee: Y^2 = X^3 + t^{360} + 1$.
	The proof  relies on the systematic transformation of points from sub-surfaces to the master surface $\mathcal{E}$. We define the mapping between the rational surfaces $\mathcal{E}_{a,b}$ (or the $K3$ surface $\mathcal{E}'$) and the master curve \eqref{shi-eq1} through specific base-change homomorphisms.

The following table summarizes the transformations applied to pull back points $(x, y)$ defined over a parameter $v$ to the points $(X,Y)$ on the Shioda surface defined over $t$.

\begin{table}[htbp]
	\centering
	\caption{Base-Change Homomorphisms for $\mathcal{E}(\mathcal{K})$}
	\label{tab:homomorphisms}
	\begin{tabular}{@{}llll@{}}
		\toprule
		\textbf{Sub-surface} & \textbf{Relation} & \textbf{X-Coordinate} & \textbf{Y-Coordinate} \\ \midrule
		$\mathcal{E}_{2,1}$ & $v = t^{360}$ & $X = x(t^{360})/t^{240}$ & $Y = y(t^{360})/t^{360}$ \\
		$\mathcal{E}_{3,1}$ & $v = t^{360}$ & $X = x(t^{360})/t^{360}$ & $Y = y(t^{360})/t^{540}$ \\
		$\mathcal{E}_{1,2}$ & $v = t^{180}$ & $X = x(t^{180})/t^{60}$ & $Y = y(t^{180})/t^{90}$ \\
		$\mathcal{E}_{2,2}$ & $v = t^{180}$ & $X = x(t^{180})/t^{120}$ & $Y = y(t^{180})/t^{180}$ \\
		$\mathcal{E}_{1,3}$ & $v = t^{120}$ & $X = x(t^{120})/t^{80}$ & $Y = y(t^{120})/t^{120}$ \\
		$\mathcal{E}_{1,4}$ & $v = t^{90}$ & $X = x(t^{90})/t^{30}$ & $Y = y(t^{90})/t^{45}$ \\
		$\mathcal{E}_{0,5}$ & $v = t^{72}$ & $X = x(t^{72})$ & $Y = y(t^{72})$ \\
		$\mathcal{E}_{1,5}$ & $v = t^{72}$ & $X = x(t^{72})/t^{24}$ & $Y = y(t^{72})//t^{36}$ \\
		$\mathcal{E}' (K3)$ & $v = t^{36}$ & $X = x(t^{36})/t^{12}$ & $Y = y(t^{36})/t^{18}$ \\ \bottomrule
	\end{tabular}
	\label{Tab2}
\end{table}

\subsection{Computational Construction of the $68 \times 68$ Gram Matrix}

The verification of the Mordell-Weil rank of Shioda's elliptic surface $\mathcal{E}: Y^2 = X^3 + t^{360} + 1$ relies on the explicit calculation of the Gram matrix $M_{\text{Shioda}}$ for the 68 generated sections $\{P_1, \dots, P_{68}\}$.

The height pairing on the master surface $\mathcal{E}$ over $\mathbb{C}(t)$ is related to the pairing on the sub-surfaces $\mathcal{E}_{a,b}$ and $\mathcal{E}'$ over $\mathbb{C}(v)$ via the degree of the base change $m = [\mathbb{C}(t) : \mathbb{C}(v)]$. For any sections $P_i, P_j \in \mathcal{E}(\mathbb{C}(t))$ derived from $Q_i, Q_j$ on a sub-surface, the pairing scales as:
\begin{equation}
	\langle P_i, P_j \rangle_{\mathcal{E}} = m \cdot \langle Q_i, Q_j \rangle_{\text{sub-surface}}
\end{equation}
This scaling is geometrically supported by the intersection multiplicities $(P \cdot \mathcal{O})$ at $t=0$ and $t=\infty$, which are induced by the rational coordinates (negative powers of $t$) in the transformation formulas. While $(Q \cdot \mathcal{O}) = 0$ for the polynomial sections on sub-surfaces, the poles at $t=0$ and $t=\infty$ precisely reconcile the arithmetic genus $\chi(\mathcal{S}_{\mathcal{E}}) = 60$ with the lower genus of the components.

Due to the orthogonality of the subspaces corresponding to the decomposition of $\mathcal{E}(\mathbb{C}(t))$, the global Gram matrix $M_{68}$ is a block-diagonal matrix. The matrix consists of eleven blocks corresponding to the scaled sub-matrices computed in Sections \ref{K3}  through \ref{A2-(2,1)}.

\begin{equation}
M_{68} = \text{diag}\left( 360 M_1, 360 M_1, 180 M_2, 180 M_2, 180 M_2', 120 M_3, 120 M_3, 90 M_4, 72 M_5, 72 M_5, 36 M_{10} \right)
\end{equation}

The determinant of $M_{68}$ is the product of the determinants of the scaled sub-blocks. For a sub-matrix $M_k$ of rank $r_k$ scaled by $m_k$, the block determinant is $m_k^{r_k} \det(M_k)$. Using the specific determinants derived for each sub-lattice:

\begin{table}[h!]
	\centering
	\begin{tabular}{@{}lllll@{}}
		\toprule
		Sections & Sub-surface & Rank ($r_k$) & Degree ($m_k$) & $\det(\text{Block})$ \\ \midrule
		$P_1 - P_4$ & $\mathcal{E}_{2,1}, \mathcal{E}_{3,1}$ & $2 \times 2$ & $360$ & $(360^2 \cdot \frac{1}{12})^2 = 360^4 \cdot \frac{1}{144}$ \\
		$P_5 - P_{12}$ & $\mathcal{E}_{1,2}, \mathcal{E}_{3,2}$ & $2 \times 4$ & $180$ & $(180^4 \cdot \frac{1}{4})^2 = 180^8 \cdot \frac{1}{16}$ \\
		$P_{13} - P_{16}$ & $\mathcal{E}_{2,2}$ & 4 & $180$ & $180^4 \cdot \frac{1}{9}$ \\
		$P_{17} - P_{28}$ & $\mathcal{E}_{1,3}, \mathcal{E}_{2,3}$ & $2 \times 6$ & $120$ & $(120^6 \cdot \frac{1}{3})^2 = 120^{12} \cdot \frac{1}{9}$ \\
		$P_{29} - P_{36}$ & $\mathcal{E}_{1,4}$ & 8 & $90$ & $90^8 \cdot 1$ \\
		$P_{37} - P_{52}$ & $\mathcal{E}_{0,5}, \mathcal{E}_{1,5}$ & $2 \times 8$ & $72$ & $(72^8 \cdot 1)^2 = 72^{16}$ \\
		$P_{53} - P_{68}$ & $\mathcal{E}^{\prime}$ (K3) & 16 & $36$ & $36^{16} \cdot 5^4$ \\ \bottomrule
	\end{tabular}
	\caption{Determinant contributions of sub-lattice blocks to $M_{\text{Shioda}}$.}
\end{table}

The final determinant of the $68 \times 68$ matrix is:
\begin{equation}
	\det(M_{68}) = \frac{360^4 \cdot 180^{12} \cdot 120^{12} \cdot 90^8 \cdot 72^{16} \cdot 36^{16} \cdot 5^4}{144 \cdot 16 \cdot 9 \cdot 9}
\end{equation}

Since $\det(M_{68}) > 0$, the 68 sections are linearly independent over $\mathbb{Q}$, confirming that $rk(\mathcal{E}(\mathbb{C}(t))) = 68$.

	\subsection{Computational Complexity}

The determination of the splitting field $\mathcal{K}$ represents a significant symbolic computation task. The field is constructed as the compositum of splitting fields $\mathcal{K}_b$ calculated for each sub-surface. 
More precisely,  the fields $\KK_1$,  $\KK_2$,  $\KK'_2$, $\KK_3$, $\KK_4$, $\KK_5$ and $\KK'$ that are defined by polynomials of degree 2, 16, 3, 27, 48, 120, and 196, provided in \cite[Gp-360]{Shioda360-Codes}.

The primary computational hurdles included:
\begin{itemize}
	\item \textbf{Degree of Extensions}: The final splitting field is defined by two primary minimal polynomials of degree $1728$ and $5760$, given in  \cite[Pol-1728]{Shioda360-Codes} and   \cite[Pol-5760]{Shioda360-Codes}.
	\item \textbf{Polynomial Resultants}: Finding the fundamental polynomials $\Phi(U)$ involved eliminating variables from high-degree ideals, requiring optimized algorithms in Maple.
	\item \textbf{Field Composita}: The use of the \textsf{Pari/GP} command \texttt{polcompositum} for fields of degree $D > 1000$ required substantial memory resources, as intermediate calculations involved polynomials with coefficients spanning several thousand digits.
\end{itemize}

All symbolic data, including exact coordinates for the $68$ points and the coefficients of the minimal polynomials for $\mathcal{K}$, are provided in \cite[Points-68]{Shioda360-Codes}.

\bibliographystyle{amsplain} 

\bibliography{SBIB1}{} 

@manual{PARI2,
      organization = "{The PARI~Group}",
      title        = "{PARI/GP version \texttt{2.17.3}}",
      year         = 2025,
      address      = "Univ. Bordeaux",
      note         = "available from \url{http://pari.math.u-bordeaux.fr/}"
    }

@manual{maple,
  title        = {Maple},
  author       = {{Maplesoft, a division of Waterloo Maple Inc.}},
  organization = {Maplesoft},
  address      = {Waterloo, Ontario},
  year         = {2024},
  note         = {Version 2024},
  url          = {https://www.maplesoft.com/}
}

@misc{Shioda360-Codes,
  author = {Salami, Sajad},
  title = {Checking codes for computations in this paper},
  year = {2025},
  publisher = {\url{https://github.com/sajadsalami/Shioda360-Codes}},
  journal = {GitHub repository},
  commit = {}
}

@misc{Shioda-Codes,
  author = {Salami, Sajad},
  title = {Checking codes for computations: The splitting field and generators of Shioda's elliptic surface  \(y^2=x^3+ t^m+1\)  (I)}}

@misc{k3-codes,
  author = {Salami, Sajad},
  title = {Checking codes for computations: {Generators and splitting fields of certain elliptic K3 surfaces}},
  year = {2025},
  publisher = {\url{https://github.com/sajadsalami/K3Surface-Check-Codes}},
  journal = {GitHub repository},
  commit = {}
}

@misc{SS-Shz-2025,
  title = {The splitting field and generators of Shioda's elliptic surface  \(y^2=x^3+ t^m+1\)  (I)},
  author = {Salami, S. and Shamsi Zargar, A.},
  year = {2025},
  eprint = {2512.16578},
  archivePrefix = {arXiv},
  primaryClass = {math.NT},
  url = {https://arxiv.org/abs/2512.16578},
  note = {Preprint}
}

@book{movasati2021course,
  title={A Course in Hodge Theory: With Emphasis on Multiple Integrals},
  author={Movasati, Hossein},
  isbn={9781571464012},
  series={Surveys of Modern Mathematics},
  volume={18},
  year={2021},
  publisher={International Press of Boston},
  address={Somerville, MA}
}

@misc{salami2022generators,
  title = {Generators and splitting field of certain elliptic {K}3 surfaces},
  author = {Salami, S. and Shamsi Zargar, A.},
  year = {2022},
  eprint = {2206.05372},
  archivePrefix = {arXiv},
  primaryClass = {math.NT},
  url = {https://arxiv.org/abs/2206.05372},
  note = {Preprint}
}

@Book{Schuett2019,
  author    = {Matthias {Sch\"utt} and Tetsuji {Shioda}},
  publisher = {Singapore: Springer},
  title     = {Mordell-Weil lattices},
  year      = {2019},
  isbn      = {978-981-329-300-7; 978-981-329-303-8; 978-981-329-301-4},
  volume    = {70},
  doi       = {10.1007/978-981-329-301-4},
  issn      = {0071-1136},
  journal   = {Ergebnisse der Mathematik und ihrer Grenzgebiete. 3. Folge},
  language  = {English},
  pages     = {xvi + 431},
  zbl       = {1433.14002},
}

@InCollection{Shioda1999,
  author    = {Tetsuji {Shioda}},
  booktitle = {Integral quadratic forms and lattices. Proceedings of the international conference on integral quadratic forms and lattices, Seoul National University, Seoul, Korea, June 15--19, 1998. Dedicated to the memory of Dennis Ray Estes},
  publisher = {Providence, RI: American Mathematical Society},
  title     = {Cyclotomic analogue in the theory of algebraic equations of type \(E_6\), \(E_7\), \(E_8\)},
  year      = {1999},
  isbn      = {0-8218-1949-6},
  pages     = {87--96},
  language  = {English},
  zbl       = {0955.11016},
}

@InCollection{Shioda1999a,
  author    = {Tetsuji {Shioda}},
  booktitle = {Algebraic geometry: Hirzebruch 70. Proceedings of the algebraic geometry conference in honor of F. Hirzebruch's 70th birthday, Stefan Banach International Mathematical Center, Warszawa, Poland, May 11--16, 1998},
  publisher = {Providence, RI: American Mathematical Society},
  title     = {The splitting field of Mordell-Weil lattices},
  year      = {1999},
  isbn      = {0-8218-1149-5},
  pages     = {297--303},
  language  = {English},
  zbl       = {0995.11038},
}

@InCollection{Shioda1992a,
  author    = {Tetsuji {Shioda}},
  booktitle = {Journées arithmétiques. Exposés présentés aux dix-septièmes congrès à Genève, Suisse, 9-13 septembre 1991},
  publisher = {Paris: Soci\'et\'e Ma\-th\'e\-ma\-tique de France},
  title     = {Some remarks on elliptic curves over function fields},
  year      = {1992},
  pages     = {99--114},
  language  = {English},
  zbl       = {0820.14016},
}

@Article{Shioda1991d,
  author    = {Tetsuji {Shioda}},
  journal   = {Journal of the Mathematical Society of Japan},
  title     = {Construction of elliptic curves with high rank via the invariants of the Weyl groups},
  year      = {1991},
  issn      = {0025-5645},
  number    = {4},
  pages     = {673--719},
  volume    = {43},
  doi       = {10.2969/jmsj/04340673},
  language  = {English},
  publisher = {Mathematical Society of Japan, Tokyo},
  zbl       = {0751.14018},
}

@Article{Chahal2000,
  author    = {Jasbir {Chahal} and Matthijs {Meijer} and Jaap {Top}},
  journal   = {Commentarii Mathematici Universitatis Sancti Pauli},
  title     = {Sections on certain \(j=0\) elliptic surfaces},
  year      = {2000},
  issn      = {0010-258X},
  number    = {1},
  pages     = {79--89},
  volume    = {49},
  language  = {English},
  publisher = {Rikkyo University (St. Paul's University), Department of Mathematics, Tokyo},
  zbl       = {0972.11051},
}

@Article{Usui2008,
  author    = {Hisashi {Usui}},
  journal   = {Commentarii Mathematici Universitatis Sancti Pauli},
  title     = {On the Mordell-Weil lattice of the elliptic curve \(y^2=x^3+t^m+1\). IV},
  year      = {2008},
  issn      = {0010-258X},
  number    = {1},
  pages     = {23--63},
  volume    = {57},
  language  = {English},
  publisher = {Rikkyo University (St. Paul's University), Department of Mathematics, Tokyo},
  zbl       = {1158.11028},
}

@mastersthesis{meijer,
  title={High rank elliptic surfaces},
  author={Meijer, Matthijs},
  school={University of Groningen},
  year={1999},
  address={P.O. Box 800, 9700 AV Groningen},
  note={Department of Mathematics}
}

\end{document}